\documentclass[10pt,a4paper]{article}

\usepackage{graphicx} 
\usepackage{tikz}
\usetikzlibrary{hobby,calc,positioning}
\usetikzlibrary{decorations.markings}
\tikzstyle{vertex}=[draw=black,fill = black,inner sep = 1pt,circle]
\usepackage[dvipsnames,svgnames,table]{xcolor}

\usepackage{amssymb}
\usepackage{amsthm}
\usepackage{amsmath}
\usepackage{enumitem}

\usepackage[hyphens]{url}
\usepackage[linktoc=all,hidelinks,colorlinks,unicode=true]{hyperref} 
\usepackage[capitalise,compress,nameinlink,noabbrev]{cleveref} 
\hypersetup{linkcolor={blue!70!black},citecolor={red!70!black},urlcolor={purple!70!black}}

\usepackage{todonotes}
\usepackage{subcaption}

\theoremstyle{plain}
\newtheorem{theorem}{Theorem}
\newtheorem{lemma}[theorem]{Lemma}
\newtheorem{problem}{Problem}

\newtheorem{claim}{Claim}[theorem]

\newtheorem{observation}[theorem]{Observation}

\theoremstyle{definition}

\newcommand{\defin}[1]{\emph{\textcolor{ForestGreen}{#1}}}

\newenvironment{poc}{\begin{proof}[Proof of
    Claim]}{\end{proof}}

\renewcommand{\vec}[1]{#1}

\newcommand{\T}{\mathcal{T}}

\newcommand{\FD}[2]{\vec{\mathcal{F}}_{#2}(#1)}
\DeclareMathOperator{\dist}{dist}

\title{A Gray code for arborescences of tournaments}
\author{
  Marthe Bonamy \footnotemark[1] \and
  Michael Hoffmann \footnotemark[2] \and
  Clément Legrand-Duchesne \footnotemark[3] \and
  Günter Rote \footnotemark[4]  
}

\begin{document}

\maketitle

\footnotetext[1]{CNRS, LaBRI, Université de Bordeaux, France (\textsf{\href{mailto:marthe.bonamy@u-bordeaux.fr}{marthe.bonamy@u-bordeaux.fr}})}
\footnotetext[2]{Department of Computer Science, ETH Zürich, Switzerland (\textsf{\href{mailto:hoffmann@inf.ethz.ch}{hoffmann@inf.ethz.ch}})}
\footnotetext[3]{Theoretical Computer Science Department, Faculty of Mathematics and Computer Science, Jagiellonian University, Kraków, Poland
(\textsf{\href{mailto:clement.legrand-duchesne@uj.edu.pl}{clement.legrand-duchesne@uj.edu.pl}}).}
\footnotetext[4]{Institut für Informatik, Freie Universität Berlin
(\textsf{\href{mailto:rote@inf.fu-berlin.de}{\nolinkurl{rote@inf.fu-berlin.de}}}).}

\begin{abstract}
    We consider the following question of Knuth: given a directed graph $G$ and a root $r$, can the arborescences of $G$ rooted in $r$ be listed such that any two consecutive arborescences differ by only one arc? Such an ordering is called a pivot Gray code and can be formulated as a Hamiltonian path in the reconfiguration graph of the arborescences of $G$ under arc flips, also called flip graph of $G$. We give a positive answer for tournaments and explore several conditions showing that the flip graph of a directed graph may contain no Hamiltonian cycles.
\end{abstract}

A \defin{Gray code} is a linear or cyclic order on the elements of a fixed set
(usually bit representations of numbers between 0 and $2^n-1$), such that
consecutive elements differ on exactly one bit. Gray codes were originally
considered to avoid the errors introduced by unperfectly synchronised physical
switches, causing a period of transition between two consecutive binary numbers.
A Gray code can also be seen as a Hamiltonian path or even a Hamiltonian cycle in
the reconfiguration graph: the graph whose vertices are the enumerated objects,
with an edge between any two objects at Hamming distance one (see Mütze's
extensive survey on Gray codes for combinatorial
structures~\cite{mutze2023Gray}).

In the context of spanning trees of a graph $G$, two types of Gray codes can be
considered. The most general one is an order with the
\defin{revolving door} property: each spanning tree in the sequence is obtained
from the previous one by an edge exchange, that is, by removing an edge and
adding another one. The first such algorithm was given by Cummins in
1966~\cite{cummins1966Hamilton}. In fact, Cummins showed that the
reconfiguration graph of spanning trees under edge exchanges of any graph $G$,
also called \defin{flip graph of $G$}, is edge-Hamiltonian. Namely, for each edge
of the flip graph there exists a Hamiltonian cycle passing through this
edge. Shank~\cite{shank1968Note} then gave a short and simple proof of this
result. As of today, the most efficient known Gray code enumeration algorithm for spanning trees runs
in constant delay between consecutive outputs on average and was given by
Smith~\cite{smith1997Generating}. Finally, the set of spanning trees of a graph
$G$ are the bases of the matroid formed by the edges of $G$ and Holzmann and
Harary~\cite{holzmann1972Tree} generalised Shank's proof to the reconfiguration
graph of the bases of any matroid under element exchanges.

The second type of Gray codes for spanning trees are those with the \defin{strong
  revolving door} property, that is, those in which the edges added and deleted
at each step share a common endpoint. These Gray codes are also referred to as
\defin{pivot Gray code}.

\begin{problem}
  \label{prob:pivot_undirected}
  Does every graph $G$ admit a pivot Gray code on its spanning trees?
\end{problem}

Although it is rather simple to prove that the corresponding reconfiguration
graphs is connected for any graph $G$, it remains open whether all graphs admit
a pivot Gray code for spanning trees. Indeed, one of the main techniques used to
construct a Gray code for spanning trees consists in partitioning the spanning
trees into two sets: those containing a specific edge $e$ and those avoiding
it. One can then apply induction on both sets by considering the graph $G/e$
where $e$ is contracted, and the graph $G-e$ where $e$ is removed,
respectively. However, the uncontraction of the edge $e$ does not preserve the
strong revolving door property, as the edges incident to $e$ are 
adjacent in $G/e$ but not in~$G$.  Nevertheless, pivot Gray codes for spanning
trees have been constructed in some graph classes, namely for fan
graphs~\cite{CGS-2024-pivot} and more generally for outerplanar
graphs~\cite{behrooznia2024Listing}.

As Knuth noted~\cite[Answer to Exercise 7.2.1.6–102]{kn4a},
\cref{prob:pivot_undirected} would immediately follow from the existence of a
Gray code on the rooted arborescences of a directed graph, by replacing each edge by
two arcs going in both directions. More precisely, an \defin{arborescence} of a
directed graph $G$ rooted in $r$ is a spanning tree of $G$ directed away
from $r$. Given two arborescences that differ in exactly one arc, the arcs added
and deleted must point to the same vertex $u$, for this vertex $u$ to be
accessible in both arborescences from $r$. Therefore, Gray codes on the
arborescences of $G$ naturally have the strong revolving door property. The
problem of designing such a Gray code was proposed by Knuth, with an estimated
difficulty of 46/50:

\begin{problem}[Exercise 7.2.1.6–102 in~\cite{kn4a}]\label{prob:pivot_directed}
  Does every directed graph $G$ admit a pivot Gray code on its arborescences
  rooted in a given vertex?
\end{problem}

In 1967, Chen~\cite[Theorem~1]{chen-1967} claimed that for any directed graph
$G$ and vertex $r\in V(G)$, the flip graph on the arborescences of $G$ rooted in $r$ contains a Hamiltonian cycle, provided
 there are at least three arborescences. A counterexample to this claim was
provided by Rao and Raju~\cite{rao-raju-1972} in 1972: they constructed a family
of directed graphs whose flip graph is a path. These digraphs are not just a
sporadic exception: we characterise the cases where an arborescence has degree 1
in the flip graph. In~\cref{sec:bipartite}, we construct a greater
variety of counterexamples of directed graphs with unbalanced bipartite flip
graphs, hence without Hamiltonian cycle.
However, as we will show, the imbalance 
for
these counterexamples cannot exceed one; thus they do not contradict the existence
of a Hamiltonian \emph{path} in the flip graph.
Finally, our main
result is the construction of a pivot Gray code on the arborescences of any
tournament (see~\cref{sec:tournaments}):
\begin{theorem}\label{thm:tournaments}
  Let $G$ be a tournament and $r$ be a vertex of $G$.
  The flip graph of the arborescences of $G$ rooted in $r$ admits a Hamiltonian path.
\end{theorem}

\section{Preliminaries}
\subsection{Notations and glossary}
Given a directed graph $G=(V,E)$, we denote by $u \to v$ or $uv$ an arc going from
$u$ to $v$.
We denote by
$N^+(v) = \{u \colon v \to u\}$ the
\defin{outneighbourhood} of $v$ and $\deg^+(v) = |N^+(v)|$ its
\defin{outdegree}.
Analogously, we denote by $N^-(v) = \{u \colon u \to v\}$ its \defin{inneighbourhood} and
$\deg^-(v) = |N^-(v)|$ its \defin{indegree}. A vertex of outdegree zero is a
\defin{sink}. The \defin{support} of a directed (multi)-graph $G$ is the undirected graph
on the same vertex set, with an edge between $u$ and $v$ if $G$ has at least one arc from $u$ to $v$
or vice versa. 
Given a fixed vertex $r$ of $G$ called
\defin{root}, we call \defin{arborescence} any spanning directed tree of $\vec{G}$
in which all arcs are oriented away from $r$. Two arborescences that are equal
on all arcs but one are said to differ by an \defin{arc flip}. \defin{Flipping in}
the arc $uv$ in an arborescence $\vec{T}$ corresponds to removing the unique arc
entering $v$ and replacing it by the arc $uv$, thereby obtaining another
arborescence that differs from $\vec{T}$ by an arc flip. The \defin{flip graph}
of $\vec{G}$ rooted in $r$, is the undirected graph $\FD{\vec{G}}{r}$ whose
nodes are the arborescences of $\vec{G}$ rooted in $r$, with an edge between each
pair of arborescences that differ by an arc flip.

Given a directed graph $G$ rooted in $r$ and one of its vertices $u$, denote
$D_G(u)$ the \defin{descendants} of $u$, that is the vertices $v$ such that
all paths from $r$ to $v$ pass through~$u$ ($u$ included).
Given an arborescence
$A$ of $G$, an arc $uv \in E(G) \setminus E(A)$ can be flipped in if and only if
$v$ is not a descendant of $u$ in $A$. We now prove an auxiliary lemma allowing
us to extend subarborescences into arborescences. 

\begin{lemma}\label{lem:completion}
  Let $\vec G$ be a directed graph and $r$ a vertex of $G$, such that $G$ admits at least one
  arborescence rooted in $r$. Any directed subtree $T$ of $\vec G$ rooted in $r$ can be
  completed into an arborescence.
\end{lemma}
\begin{proof}
  We proceed by induction of $|V(T)|$. Let $A$ be an arborescence of $\vec G$
  and $T$ a directed subtree of $\vec {G}$, both rooted in $r$. Let $X$ be the
  set of vertices of $G$ that do not belong to $T$. Let $P$ be a minimal path
  from $r$ to a vertex in $X$, such a path exists because $A$ is an arborescence
  rooted in $r$. Let $v$ and $u$ be the last and penultimate vertices in
  $P$. We have $u \to v$ and $u$ belongs to $V(T)$ by minimality of $P$;
  therefore, adding the arc $uv$ to $T$ extends $T$. 
\end{proof}

\subsection{Hamiltonicity of ladders and hypercubes}
We start with some easy constructions of Hamiltonian paths in ladders and
hypercubes, which will be used in other constructions. A \defin{ladder} of length $n$ is the undirected graph consisting of
two paths $a_1, \dots,a_n$ and $b_1, \dots,b_n$ of length $n$, with an
additional edge between $a_i$ and $b_i$ for all $i$. A vertex in the ladder has
\defin{level} $i$ if $v \in \{a_i,b_i\}$.

\begin{lemma}\label{lem:ladder}
  Let $F$ be the ladder of length
  $n$, and
 let $i,j\le n$ be integers with $i \neq j$.
   Let $u$ be a vertex of level $i$. Then  $F$ has a Hamiltonian path
 from $u$ to some vertex at level $j$.
\end{lemma}
Note that we cannot prescribe \emph{which} of the two vertices at level $j$
is the endpoint of this Hamiltonian path. In fact, since $F$ is bipartite, a parity argument shows that only one of the two vertices at level $j$ can potentially
be the endpoint of a Hamiltonian path starting at~$u$.
\begin{proof}
  Without loss of generality, assume that $j>i$ and $u = a_i$. Consider the paths
  $P_1$, $P_2$ and $P_3$, where $P_1 = a_i, \dots,a_1,b_1, \dots,b_i$ and $P_3 =
  b_j, \dots,b_n, a_n, \dots,a_j$, and $P_2$ goes from $b_i$ to some vertex $v
  \in \{a_j, b_j\}$ by zigzagging (see \cref{fig:ladder}):
  $$P_2 = b_i, b_{i+1}, a_{i+1}, a_{i+2}, b_{i+2}, b_{i+3}, \dots,v$$
 Concatenating the three paths $P_1$, $P_2$, and $P_3$ results in a Hamiltonian
  cycle, ending at a vertex $w \neq v$ of level $j$. Note that $w$ depends on
  the parity of $j-i$. 
\end{proof}
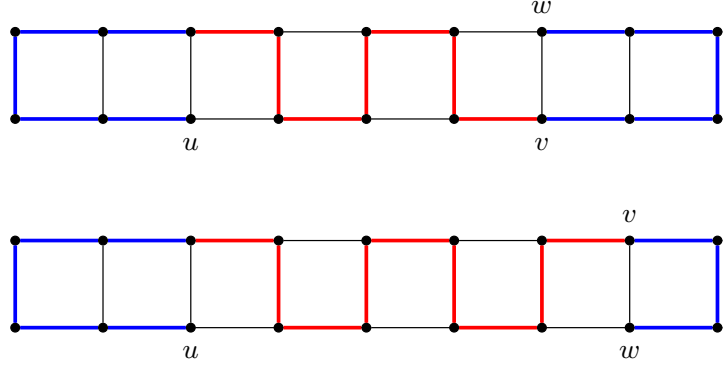
\begin{figure}[h!]
  \begin{subfigure}[t]{\textwidth}
    \centering
    \begin{tikzpicture}
      \foreach \i in {1,...,9}{
        \node[inner sep=1pt,circle, draw=black,fill=black] (x\i) at (\i,0) {};
        \node[inner sep=1pt,circle, draw=black,fill=black] (y\i) at (\i,1) {};
        \draw (x\i) -- (y\i);
      }

      \foreach \i in {1,...,8}{
        \pgfmathsetmacro{\j}{\i+1}
        \draw (x\i) -- (x\j) (y\i) -- (y\j);
      }

      \draw[very thick, blue] (x3) -- (x2) -- (x1) -- (y1) -- (y2) -- (y3);
      \draw[very thick, red] (y3) -- (y4) -- (x4) -- (x5) -- (y5) -- (y6) -- (x6) --
      (x7);
      \draw[very thick, blue] (x7) -- (x8) -- (x9) -- (y9) -- (y8) -- (y7);

      \node[below=1pt of x3] {\small $u$};
      \node[below=1pt of x7] {\small $v$};
      \node[above=1pt of y7] {\small $w$};
    \end{tikzpicture}
  \end{subfigure}\\
  
  \begin{subfigure}[t]{\textwidth}
    \centering
    \begin{tikzpicture}
      \foreach \i in {1,...,9}{
        \node[inner sep=1pt,circle, draw=black,fill=black] (x\i) at (\i,0) {};
        \node[inner sep=1pt,circle, draw=black,fill=black] (y\i) at (\i,1) {};
        \draw (x\i) -- (y\i);
      }

      \foreach \i in {1,...,8}{
        \pgfmathsetmacro{\j}{\i+1}
        \draw (x\i) -- (x\j) (y\i) -- (y\j);
      }

      \draw[very thick,blue] (x3) -- (x2) -- (x1) -- (y1) -- (y2) -- (y3);
      \draw[very thick,red] (y3) -- (y4) -- (x4) -- (x5) -- (y5) -- (y6) -- (x6)
      -- (x7) -- (y7) -- (y8);
      \draw[very thick,blue] (y8) -- (y9) -- (x9) -- (x8);

      \node[below=1pt of x3] {\small $u$};
      \node[below=1pt of x8] {\small $w$};
      \node[above=1pt of y8] {\small $v$};

    \end{tikzpicture}
  \end{subfigure}
  \caption{Two Hamiltonian paths in a ladder with extremities at different
    levels. The paths $P_1$ and $P_3$ are drawn in blue, the path $P_2$ in red.}
  \label{fig:ladder}
\end{figure}

The \defin{hypercube $K_2^d$} of dimension $d$ is the graph on $2^d$ vertices indexed by
$\{0,1\}^d$ in which two vertices are adjacent if they differ on exactly one
coordinate. Hypercubes of dimension at least 2 are edge-Hamiltonian:

\begin{lemma}\label{lem:hypercube}
  Let $K_2^d$ be the hypercube of dimension $d \ge 2$ and $uv \in E(G)$. The
  hypercube $K_2^d$ has a Hamiltonian cycle containing the edge $uv$.
  \qed
\end{lemma}

Listing binary strings was the original problem considered by Gray~\cite{gray1953Pulse},
and thus, many different proofs of the Hamiltonicity of the hypercube exist, such as the binary reflected Gray code. From this, one can deduce that the hypercube is edge-Hamiltonian simply by the edge-transitivity of the hypercube.

\subsection{Reductions}

We now prove several reductions rules that preserve the flip graph of a directed
graph, or at least the existence of a Hamiltonian path in it. The operations we
consider are the removal, the subdivision, the contraction and the duplication of an arc.

\paragraph{Arc deletion.}
Given a fixed directed graph $G$, we denote $G-uv$ the directed
graph obtained by removing the arc $u \to v$.

\begin{observation}\label{obs:arc_deletion}
  Let $uv$ be an arc of a directed graph $G$ rooted in $r$. The flip graph of
  $G-uv$ is isomorphic to the subgraph of $\FD{\vec G}r$ induced by the
  arborescences not containing $uv$. A direct consequence of this is that every
  arc that appears in none of the arborescences of $G$ can be removed from $G$
  without affecting its flip graph.
\end{observation}

Let $G$ be a directed graph $G$ rooted in some vertex $r$.
We will say that a directed graph $H$ is \defin{built on}
$G$ if $V(H) = V(G)$ and $E(G) \subseteq E(H)$, and 
for every arc $u\to v$ in $E(H) \setminus E(G)$,
$u$ is a descendant of $v$ in $G$. We will call such an arc $u \to v$ a
\defin{backedge} of $H$.

\begin{lemma}\label{lem:built_on}
  If $H$ is built on $G$, then $\FD{\vec H}r$ is isomorphic to $\FD{\vec G}r$.
\end{lemma}
\begin{proof}
  The arcs $uv$ of $E(H) \setminus E(G)$ do not appear in any arborescence of
  $H$. Otherwise, let $A$ be an arborescence of $H$ with
  $uv \in E(A)$. Without loss of generality, assume that $uv$ was
  chosen such that $uv$ is the only arc of $E(A) \setminus E(G)$ on the path
  from the root $r$ to $v$. Then there exists a path from $r$ to $u$ that uses
  arcs of $G$ and avoids $v$, which contradicts the fact that $u$ is a
  descendant of $v$ in $G$. Thus the arborescences of $H$ are exactly the
  arborescences of $G$, and the flip graphs of $H$ and $G$ are isomorphic by \cref{obs:arc_deletion}.
\end{proof}

\paragraph{Arc contraction.}
Denote $G/uv$ the directed graph obtained by contracting the arc $u \to v$,
that is replacing $u$ and $v$ by a vertex $w$ with
$N^-(w) = N^-(u) \cup N^-(v) \setminus \{u,v\}$ and
$N^+(w) = N^+(u) \cup N^+(v) \setminus \{u,v\}$. Arc contractions do not behave
as well as edge deletions with respect to the flip graph. Indeed, the flip graph
can change when contracting an arc, even if this arc belongs to all the
arborescences of $G$. For example, the graph $G$ on three vertices $r$, $u$ and
$v$, rooted in $r$ and containing the arcs $ru$, $uv$ and $rv$, has two
arborescences: one containing $ru$ and $uv$, the other containing $ru$ and
$rv$. However contracting $ru$ results in a single arc, hence the number of
arborescences of $G$ was not preserved. Recall also from the introduction that
another problem might occur. The reconfiguration sequences on the contracted
graph do not correspond to reconfiguration sequences in the original graph. For
example, consider the graph $G$ on four vertices $r$, $x$, $y$, $z$, rooted in
$r$, with four arcs $r \to x$, $r\to y$, $x \to z$ and $y\to z$ (see
\cref{fig:contraction_counterexamplea}). When contracting $xz$ into a single
vertex $u$, the arc $ru$ can be flipped to $yu$ (see
\cref{fig:contraction_counterexampleb}). However, $r$ and $y$ have disjoint
outneighbourhoods in $G$ so this flip does not correspond to a flip in the
original graph.
\begin{figure}[h!]
  \centering
  \begin{subfigure}{.45\textwidth}
    \centering
    \begin{tikzpicture}[decoration={
        markings,
        mark=at position 0.6 with {\arrow[scale=1.5]{>}}}
      ]
      \node[vertex,double] (r) at (0,2) {};
      \node[vertex] (x) at (-1,1) {};
      \node[vertex] (y) at (1,1) {};
      \node[vertex] (z) at (0,0) {};

      \draw[postaction={decorate}]  (r) to (x);
      \draw[postaction={decorate}]  (r) to (y);
      \draw[postaction={decorate}]  (x) to (z);
      \draw[postaction={decorate}]  (y) to (z);

      \node[above=0cm] at (r) {$r$};
      \node[left=0cm] at (x) {$x$};
      \node[right=0cm] at (y) {$y$};
      \node[below=0cm] at (z) {$z$};
    \end{tikzpicture}
    \caption{$G$}\label{fig:contraction_counterexamplea}
  \end{subfigure}
  \begin{subfigure}{.45\textwidth}
    \centering
    \begin{tikzpicture}[decoration={
        markings,
        mark=at position 0.6 with {\arrow[scale=1.5]{>}}}
      ]

      \node[vertex,double] (r) at (0,2) {};
      \node[vertex] (u) at (-.5,.5) {};
      \node[vertex] (y) at (1,1) {};

      \draw[postaction={decorate}]  (r) to (u);
      \draw[postaction={decorate}]  (r) to (y);
      \draw[postaction={decorate}]  (y) to (u);

      \node[above=0cm] at (r) {$r$};
      \node[right=0cm] at (y) {$y$};
      \node[below left=0cm] at (u) {$u$};
    \end{tikzpicture}
    \caption{$G/xz$}\label{fig:contraction_counterexampleb}
  \end{subfigure}
  \caption{A directed graph $G$ for which contracting $xz$ into $u$ creates a flip unfeasible in $G$. Throughout this article, we will consistently represent the root of our directed graphs by a circled vertex, here $r$.}
  \label{fig:contraction_counterexample}
\end{figure}
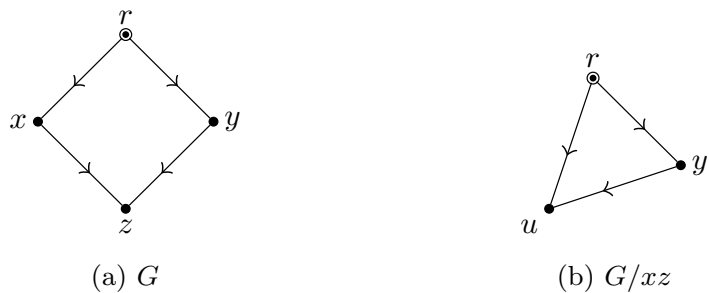

However, some contractions still preserve the existence of a Hamiltonian path.
Let $\vec{G}$ be a directed graph, $r$ its root and $rx$ an outgoing arc of
$r$. Let $\vec G'= \vec G/rx$ be the directed graph obtained after the
contraction of $rx$ into the new root $r'$. Let $\T_{/rx}$ denote the set of
arborescences of $\vec{G}$ that contain the arc $rx$. 

\begin{lemma}\label{lem:arc_contraction}
  Let $A'$ be an arborescence of $\vec G'$. If $\FD{\vec{G}'}{r'}$ contains a
  Hamiltonian path starting from $A'$, then $\FD{\vec{G}}{r}[\T_{/rx}]$ contains
  a Hamiltonian path starting from any arborescence $A$ such that 
  contracting $rx$ results in $A'$.
\end{lemma}
\begin{proof}
  Let $\phi$ be the map that associates to any arc $uv$ of $\vec{G}$ the arc $r'v$ if $u \in \{r,x\}$ and the arc $uv$ otherwise. For every arborescence $T$ rooted in $r$ and containing the arc $rx$, we have $E(T/rx) = \phi(E(T))$.

  Let $T'$ be an arborescence of $\vec{G}'$. Let $M_{T'}$ be the set of vertices in
  $N^+_G(x) \cap N^+_G(r)$ that are outneighbours of $r'$ in $T'$, and let $H_{T'}$ be the set of all arborescences $T$ of $G$ rooted in $r$ that contain the arc $rx$ and such that $\phi(T) = T/rx = T'$. We first prove that $H_{T'}$ induces a hypercube of dimension $|M_{T'}|$ in $\FD{\vec{G}}{r}$. For every arc $r'z \in E(T')$ with $z \in M_{T'}$, the arcs mapped by $\phi$ to $r'z$ are $rz$ and $xz$. For any other arc in $T'$, there is a unique arc
  in $\vec{G}$ mapped to it by $\phi$. Hence, $H_{T'}$ is in one-to-one
  correspondence with the subsets of $M_{T'}$. Moreover, two arborescences
  $T_0$ and $T_1$ in $H_{T'}$ differ by a flip if only if there exists some
  $z \in M_{T'}$ such that $rz \in T_i$ and $xz \in T_{1-i}$, and $T_0 -z = T_1
  -z$. Therefore, $H_{T'}$ induces in $\FD{\vec{G}}{r}$ a subgraph isomorphic to the Hasse diagram of the
  poset $(M_{T'}, \subset)$, i.e. a hypercube of dimension $|M_{T'}|$.

  We now show that for every $S', T'$ adjacent in $\FD{\vec{G}'}{r'}$, for every 
  arborescence $S$ in $H_{S'}$, there is a
  path in $\FD{\vec{G}}{r}[\T_{/rx}]$ starting from $S$ that visits exactly $H_{S'}$ before ending at some
  arborescence in $H_{T'}$. The lemma then follows from applying this
  iteratively on the edges of the Hamiltonian path of $\FD{\vec{G}'}{r'}$. Let
  $S \in H_{S'}$. As $H_{S'}$ induces a hypercube of
  dimension $|M_{S'}|$ in $\FD{\vec{G}}{r}$, it contains a Hamiltonian path starting at $S$ and ending at
  some $S_2 \in H_{S'}$. Let $a'b'$ be the arc flipped in in $S'$ to
  obtain $T'$. Note that $b'$ cannot be equal to $r'$, because $r'$ is the root
  of $\vec{G}'$, so $b'$ is also a vertex of $\vec G$. If
  $a' = r'$, then by definition of contraction, there exists $a \in \{r,x\}$
  such that $b' \in N^+_{\vec{G}}(a)$. Otherwise, $a'$ was not the result of the
  contraction and we set $a := a'$. In both cases, the path from $r'$ to
  $a'$ in $S'$ avoids $b'$, because the arc $a'b'$ could be flipped
  in to obtain $T'$, so there is also a path from $r$ to $a$ avoiding $b'$ in $S_2$ because
  $b' \notin \{r,x\}$. Thus, $a$ is not a descendant of $b'$ in $S_2$ and the arc
  $ab'$ can be flipped in in $S_2$. Denote $T$ the resulting arborescence. As
  $\phi(ab') = a'b'$, we have $T/rx = T'$, which concludes the proof.
\end{proof}

\paragraph{Arc subdivision.}
The following observation shows that we can restrict \cref{prob:pivot_directed}
to oriented graphs by subdiving one arc in each bigon of some directed graph.
\begin{observation}
  Subdividing an arc does not modify the flip graph.
\end{observation}
\begin{proof}
  Let $G$ be a directed graph rooted in $r$, let $G'$ be the directed graph
  obtained by subdividing an arc $uw$, and denote $v$ the vertex introduced in the
  operation. We build a bijection $\phi$ from the arborescences of $G$ to the
  arborescences of $G'$. Given an arborescence $A$ containing $uw$, let
  $\phi(A)$ be the arborescence of $G'$ obtained by subdividing $uw$ by
  introducing the vertex $v$. Given an arborescence $A$ that does not contain
  $uw$, let $E(\phi(A)) = E(A) \cup \{uv\}$. It is clear that $\phi$ is a bijection
  and that $\phi$ preserves the Hamming distance, i.e. $\dist_{\FD{G'}r}(\phi(A_1),\phi(A_2)) \le \dist_{\FD{G}r}(A_1,A_2)$ for every arborescences $A_1$ and $A_2$. Moreover, note that all
  arborescences of $G'$ contain the arc $uv$ because $v$ has indegree one, thus
  $\phi^{-1}$ also preserves the Hamming distance, which proves that $\FD{G}r$
  and $\FD{G'}r$ are isomorphic.
\end{proof}

\paragraph{Arc duplication.}
The following lemma shows that generalising \cref{prob:pivot_directed} to directed multi-graphs does not increase its difficulty.
\begin{lemma}\label{lem:arc_multiplication}
    Let $G$ be a directed multi-graph rooted in $r$. Let $e$ be an arc of $G$ and let $G'$ be the directed multi-graph obtained from $G$ by duplicating $e$, i.e. adding an arc $e'$ with same head and tail as $e$. If $\FD{G}{r}$ admits a Hamiltonian path (or cycle), then so does $\FD{G'}r$.
\end{lemma}
\begin{proof}
    For every arborescence $A$ that contains $e$, denote by $A'$ the arborescence obtained from $A$ by flipping in $e'$. Let $A$ and $B$ be two arborescences adjacent in $\FD{\vec G}r$ by flipping in an arc $f$.
If both $A$ and $B$ contain $e$, then $\{A,A',B,B'\}$ induces in $\FD{\vec G'}r$ a four-cycle: $A$ and $A'$ (resp. $B$ and $B'$) are adjacent by flipping $e$ into $e'$ and $A'$ is adjacent to $B'$ by flipping in the arc~$f$.
If only $A$ contains $e$, then $f = e$, so $\{A,A',B\}$ induces a triangle in $\FD{\vec G'}r$ because if $e$ can be flipped in in $B$, then $e'$ can also be flipped in. By replacing the edges of the Hamiltonian path (or cycle) $P = (A_1, \dots A_p)$ of $\FD{\vec G}r$ by 
appropriate paths through triangles or 4-cycles depending on whether $A_i$ and $A_{i+1}$ contain~$e$, one can build a Hamiltonian path $P'$ (or cycle) of $\FD{\vec G'}r$ such that,
whenever $P$ visits an arborescence $A_i$ with $e \in A_i$, 
$P'$ visits $A_i$ and $A_i'$ 
(see \cref{fig:multiplication}).
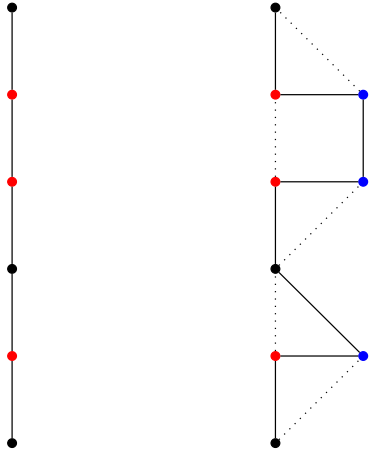
\begin{figure}[h!]
    \centering
    \begin{tikzpicture}
        \foreach \i in {1,3,6}{
            \node[inner sep=1pt,circle, draw=black,fill=black] (a\i) at (0,\i){};
        }
        \foreach \i in {2,4,5}{
            \node[inner sep=1pt,circle, draw=red,fill=red] (a\i) at (0,\i){};
        }
        \foreach \i in {1,..., 5}{
            \pgfmathsetmacro{\ii}{int(\i +1)}
            \draw (a\i) -- (a\ii);
        }

        \begin{scope}[shift = {(3,0)}]
        \foreach \i in {1,3,6}{
            \node[inner sep=1pt,circle, draw=black,fill=black] (b\i) at (0,\i){};
        }
        \foreach \i in {2,4,5}{
            \node[inner sep=1pt,circle, draw=red,fill=red] (b\i) at (0,\i){};
            \node[inner sep=1pt,circle, draw=blue,fill=blue] (c\i) at (1,\i){};
            \draw (b\i)--(c\i);
        }

        \draw (b1) -- (b2) (c2) -- (b3) -- (b4) (c4) -- (c5) (b5) -- (b6);
        \draw[dotted] (b1) -- (c2) (b2) -- (b3) -- (c4) (b4) -- (b5) (c5) -- (b6);
        \end{scope}
    \end{tikzpicture}
    \caption{Constructing a Hamiltonian path of $\FD{\vec{G'}}r$ from a Hamiltonian path of $\FD{\vec G}r$. 
    In the Hamiltonian path of $\FD{\vec G}r$
    on the left, the vertices that correspond to arborescences containing $e$ are drawn in red. On the right a Hamiltoninan path of $\FD{\vec G'}r$. The vertices corresponding to arborescences containing $e$ are drawn in red, those containing $e'$ in blue. The dotted edges are edges that are
    present in the flip graph,
    but not 
    used in the Hamiltonian path.
    }\label{fig:multiplication}
\end{figure}
\end{proof}

\section{Directed graphs with no Hamiltonian cycle in their flip graph}
In this section, we give several counterexamples witnessing that the flip graph
of a directed graph may not contain a Hamiltonian cycle.

\subsection{Flip graph can be paths}
In \cite{rao-raju-1972}, Rao and Raju showed that the flip graph of every bidirected
cycle is a path. We recall their proof here:
\begin{lemma}[\cite{rao-raju-1972}]
  Let $\vec G$ be the directed graph obtained by replacing each edge of a
  $n$-cycle by a bigon. The flip graph of $\vec G$ is a path on $n$ vertices.
\end{lemma}
\begin{proof}
  Let $V = \{v_1, \dots, v_{n}\}$ be the vertices of $\vec G$, such that $v_1$ is
  the root and each vertex $v_i$ is incident to $v_{i-1}$ and
  $v_{i+1}$, where the indices are considered modulo $n$. There are only two paths form $v_1$ to $v_i$ in $\vec G$:
  the first one is $v_1, \dots,v_i$ and  the second $v_1, v_n, \dots,v_i$. So every
  arborescence of $\vec G$ rooted in $r$ corresponds to a unique interval
  $[1, i]$ of indices such that $v_j$ is reached through a path using the arc
  $v_1v_2$ for each $j \in [1,i]$. In each of these arborescences, only the arcs entering the leafs can be
  swapped. Thus the reconfiguration graph of $\vec G$ is path on $n$ vertices
  (see \cref{fig:flip_path}).
    \begin{figure}[ht!]
    \centering
     \begin{tikzpicture}[decoration={
        markings,
        mark=at position 0.6 with {\arrow[scale = 1.5]{>}}}
      ]

      \begin{scope}[scale =.5, shift={(0,0)}]
        \node[rectangle, minimum width = 1.6cm, minimum height = 1cm] (A) at
        (2.5,.5) {};
        \foreach \i in {1,2,3,4}{
          \coordinate (v\i) at ($(\i,0)$);
        }
        \node[vertex,double] (r) at (2.5,1){};
        \draw[postaction={decorate}] (r) to (v1);
        
        \foreach \i in {1,2,3}{
          \pgfmathsetmacro{\ii}{\i+1}
          \draw[postaction={decorate}, bend left] (v\i) to (v\ii);
        }

        \foreach \i in {1,2,3,4}{
          \node[vertex] at (v\i){};
        }
      \end{scope}

      \begin{scope}[scale =.5, shift={(4,0)}]
        \node[rectangle, minimum width = 1.6cm, minimum height = 1cm] (B) at
        (2.5,.5) {};
        \foreach \i in {1,2,3,4}{
          \coordinate (v\i) at ($(\i,0)$);
        }
        \node[vertex,double] (r) at (2.5,1){};
        \draw[postaction={decorate}] (r) to (v1);
        \draw[postaction={decorate}] (r) to (v4);
        
        \foreach \i in {1,2}{
          \pgfmathsetmacro{\ii}{\i+1}
          \draw[postaction={decorate}, bend left] (v\i) to (v\ii);
        }

        \foreach \i in {1,2,3,4}{
          \node[vertex] at (v\i){};
        }
      \end{scope}

      \begin{scope}[scale =.5, shift={(8,0)}]
        \node[rectangle, minimum width = 1.6cm, minimum height = 1cm] (C) at
        (2.5,.5) {};
        \foreach \i in {1,2,3,4}{
          \coordinate (v\i) at ($(\i,0)$);
        }
        \node[vertex,double] (r) at (2.5,1){};
        \draw[postaction={decorate}] (r) to (v1);
        \draw[postaction={decorate}] (r) to (v4);

        \foreach \i in {1}{
          \pgfmathsetmacro{\ii}{\i+1}
          \draw[postaction={decorate}, bend left] (v\i) to (v\ii);
        }

        \foreach \i in {3}{
          \pgfmathsetmacro{\ii}{\i+1}
          \draw[postaction={decorate}, bend left] (v\ii) to (v\i);
        }

        \foreach \i in {1,2,3,4}{
          \node[vertex] at (v\i){};
        }
      \end{scope}

      \begin{scope}[scale =.5, shift={(12,0)}]
        \node[rectangle, minimum width = 1.6cm, minimum height = 1cm] (D) at
        (2.5,.5) {};
        \foreach \i in {1,2,3,4}{
          \coordinate (v\i) at ($(\i,0)$);
        }
        \node[vertex,double] (r) at (2.5,1){};
        \draw[postaction={decorate}] (r) to (v1);
        \draw[postaction={decorate}] (r) to (v4);

        \foreach \i in {2,3}{
          \pgfmathsetmacro{\ii}{\i+1}
          \draw[postaction={decorate}, bend left] (v\ii) to (v\i);
        }

        \foreach \i in {1,2,3,4}{
          \node[vertex] at (v\i){};
        }
      \end{scope}

      \begin{scope}[scale =.5, shift={(16,0)}]
        \node[rectangle, minimum width = 1.6cm, minimum height = 1cm] (E) at
        (2.5,.5) {};
        \foreach \i in {1,2,3,4}{
          \coordinate (v\i) at ($(\i,0)$);
        }
        \node[vertex,double] (r) at (2.5,1){};
        \draw[postaction={decorate}] (r) to (v4);
        
        \foreach \i in {1,2,3}{
          \pgfmathsetmacro{\ii}{\i+1}
          \draw[postaction={decorate}, bend left] (v\ii) to (v\i);
        }

        \foreach \i in {1,2,3,4}{
          \node[vertex] at (v\i){};
        }
      \end{scope}

      \draw[ultra thick] (A) -- (B) -- (C) -- (D) -- (E);
    \end{tikzpicture}
    \caption{The flip graph of a bidirected 5-cycle}
    \label{fig:flip_path}
  \end{figure}
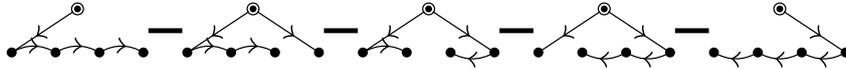
\end{proof}

\subsection{Unbalanced bipartite flip graphs}
\label{sec:bipartite}

\begin{theorem}\label{thm:bipartite}
  Let $G$ be a directed graph rooted in $r$, in which all vertices have indegree
  at most two. Then the flip graph $\FD{\vec G}r$ is bipartite.
\end{theorem}
\begin{proof}
  We can express the parity of the arborescences as a multiplicative \defin{weight}, as follows: If
  $xz$ and $yz$ are two incoming arcs of the same vertex $z$ in $G$, we arbitrarily
  assign weight $w_{wz}=+1$ to one of these arcs and $w_{yz}=-1$ to the other
  arc. Every arc that is the single incoming arc of a vertex gets weight $+1$.
  The weight $w(A)\in \{+1,-1\}$ of an arborescence $A$ is then the product of the weights of its
  arcs. Flipping an arc changes the sign of this weight, and thus the weight defines a
  bipartition of the flip graph.
\end{proof}

If the two bipartition classes differ in size by at least one, then this
constitutes yet another counterexample of a directed graph without a Hamiltonian
cycle in its flip graph. For example, the graph in \cref{fig:graph13}
has a bipartite flip graph with thirteen vertices, which is shown in
\cref{fig:flip_bipartite13}. We will also see another example with seven
arborescences
in the proof of
\cref{thm:tournaments}, see \cref{fig:flip_bipartite7}.

\begin{figure}[h!]
  \centering
  \begin{minipage}{.35\textwidth}
  \centering
  \begin{tikzpicture}[
    normalnode/.style={draw,circle, minimum size=3mm},
    rootnode/.style={draw,double,circle, minimum size=5mm},
    decoration={
      markings,
      mark=at position 0.6 with {\arrow[scale=1.5]{>}}}
    ]
    \node[rootnode] at (0,1) (v1) {1};
    \node[normalnode] at (1,0) (v2) {2};
    \node[normalnode] at (0,0) (v3) {3};
    \node[normalnode] at (-1,0) (v4) {4};
    \node[normalnode] at (0,-1) (v5) {5};

    \draw[postaction={decorate}]  (v1) -- (v2);
    \draw[postaction={decorate}]  (v1) -- (v3);
    \draw[postaction={decorate}]  (v1) -- (v4);
    \draw[postaction={decorate}]  (v2) -- (v3);
    \draw[postaction={decorate}]  (v3) -- (v4);
    \draw[postaction={decorate}]  (v4) -- (v5);
    \draw[postaction={decorate}]  (v3) -- (v5);
    \draw[postaction={decorate}]  (v5) -- (v2);
  \end{tikzpicture}
  \end{minipage}
 \hfill
  \begin{minipage}{.55\textwidth}
  \begin{displaymath}
    L = \begin{pmatrix}
    0&1&1&1&0\\
    0&0&-1&0&0\\
    0&0&0&-1&1\\
    0&0&0&0&-1\\
    0&-1&0&0&0\\
  \end{pmatrix}
  \end{displaymath}
  \end{minipage}

  \caption{A graph with 13 arborescences and no Hamiltonian cycle in its flip graph. On the right, its weighted Laplace matrix (defined in \cref{thm:balanced_bipartition}), where the $i$-th row and column correspond to the vertex labelled $i$.}
  \label{fig:graph13}
\end{figure}

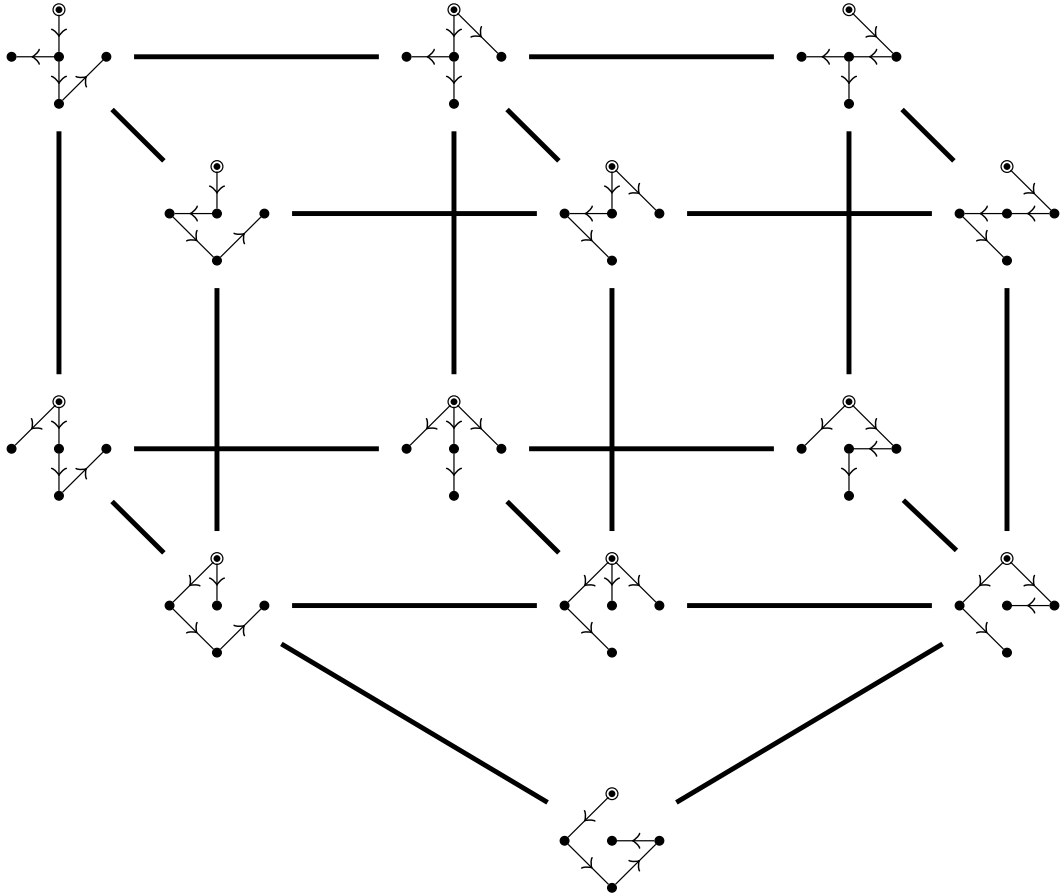
\begin{figure}[h!]
  \centering
  \begin{tikzpicture}[scale=.9,decoration={
      markings,
      mark=at position 0.6 with {\arrow[scale=1.5]{>}}}
    ]
    \foreach \i/\x/\y in {
      1/0/0,2/5/0,3/10/0,
      4/2/-2,5/7/-2,6/12/-2,
      7/0/-5,8/5/-5,9/10/-5,
      10/2/-7,11/7/-7,12/12/-7,
      13/7/-10}
    {
      \begin{scope}[shift={(\x,\y)}, scale=.6]
        \node[anchor=center,circle,inner sep=.6cm] (A\i) at (0,0) {};
        \node[vertex,double] at (0,1) (v\i1) {};
        \node[vertex] at (1,0) (v\i2) {};
        \node[vertex] at (0,0) (v\i3) {};
        \node[vertex] at (-1,0) (v\i4) {};
        \node[vertex] at (0,-1) (v\i5) {};
      \end{scope}arborescences
    }

    \draw[ultra thick] (A1) -- (A2) -- (A3);
    \draw[ultra thick] (A4) -- (A5) -- (A6);
    \draw[ultra thick] (A7) -- (A8) -- (A9);
    \draw[ultra thick] (A10) -- (A11) -- (A12);
    \foreach \i in {1,2,3,7,8,9}{
      \pgfmathsetmacro{\j}{int(\i+3)};
      \draw[ultra thick] (A\i) -- (A\j);
    }
    \foreach \i in {1,...,6}{
      \pgfmathsetmacro{\j}{int(\i+6)};
      \draw[ultra thick] (A\i) -- (A\j);
    }
    \draw[ultra thick] (A10) -- (A13) -- (A12);
    
    \draw[postaction={decorate}]  (v11) -- (v13);    
    \draw[postaction={decorate}]  (v13) -- (v14);    
    \draw[postaction={decorate}]  (v13) -- (v15);    
    \draw[postaction={decorate}]  (v15) -- (v12);
arborescences
    \draw[postaction={decorate}]  (v21) -- (v22);
    \draw[postaction={decorate}]  (v21) -- (v23);
    \draw[postaction={decorate}]  (v23) -- (v24);
    \draw[postaction={decorate}]  (v23) -- (v25);

    \draw[postaction={decorate}]  (v31) -- (v32);
    \draw[postaction={decorate}]  (v32) -- (v33);
    \draw[postaction={decorate}]  (v33) -- (v34);
    \draw[postaction={decorate}]  (v33) -- (v35);

    \draw[postaction={decorate}]  (v41) -- (v43);
    \draw[postaction={decorate}]  (v43) -- (v44);
    \draw[postaction={decorate}]  (v44) -- (v45);
    \draw[postaction={decorate}]  (v45) -- (v42);

    \draw[postaction={decorate}]  (v51) -- (v52);
    \draw[postaction={decorate}]  (v51) -- (v53);
    \draw[postaction={decorate}]  (v53) -- (v54);
    \draw[postaction={decorate}]  (v54) -- (v55);

    \draw[postaction={decorate}]  (v61) -- (v62);
    \draw[postaction={decorate}]  (v62) -- (v63);
    \draw[postaction={decorate}]  (v63) -- (v64);
    \draw[postaction={decorate}]  (v64) -- (v65);

    \draw[postaction={decorate}]  (v71) -- (v73);
    \draw[postaction={decorate}]  (v71) -- (v74);
    \draw[postaction={decorate}]  (v73) -- (v75);
    \draw[postaction={decorate}]  (v75) -- (v72);

    \draw[postaction={decorate}]  (v81) -- (v82);
    \draw[postaction={decorate}]  (v81) -- (v83);arborescences
    \draw[postaction={decorate}]  (v81) -- (v84);
    \draw[postaction={decorate}]  (v83) -- (v85);
    
    \draw[postaction={decorate}]  (v91) -- (v92);
    \draw[postaction={decorate}]  (v91) -- (v94);
    \draw[postaction={decorate}]  (v92) -- (v93);
    \draw[postaction={decorate}]  (v93) -- (v95);

    \draw[postaction={decorate}]  (v101) -- (v103);
    \draw[postaction={decorate}]  (v101) -- (v104);
    \draw[postaction={decorate}]  (v104) -- (v105);
    \draw[postaction={decorate}]  (v105) -- (v102);

    \draw[postaction={decorate}]  (v111) -- (v112);
    \draw[postaction={decorate}]  (v111) -- (v113);
    \draw[postaction={decorate}]  (v111) -- (v114);
    \draw[postaction={decorate}]  (v114) -- (v115);

    \draw[postaction={decorate}]  (v121) -- (v122);
    \draw[postaction={decorate}]  (v121) -- (v124);
    \draw[postaction={decorate}]  (v122) -- (v123);
    \draw[postaction={decorate}]  (v124) -- (v125);

    \draw[postaction={decorate}]  (v131) -- (v134);
    \draw[postaction={decorate}]  (v132) -- (v133);
    \draw[postaction={decorate}]  (v134) -- (v135);
    \draw[postaction={decorate}]  (v135) -- (v132);

  \end{tikzpicture}
  \caption{The flip graph of the graph drawn on \cref{fig:graph13}}\label{fig:flip_bipartite13}
\end{figure}

If the two bipartition classes differ in size by at least 2, this
would
immediately be
an obstacle to the existence of a Hamiltonian path in the flip path. However, we
can prove that in the flip graph of a directed graph with indegree at most two,
the size of the two parts of the bipartition differs by at most 1
(and it is easy to test whether they differ by 1, i.e. the total number of trees
is odd).

\begin{theorem}\label{thm:balanced_bipartition}
  Let $G$ be a directed graph rooted in $r$, in which all vertices have indegree
  at most two. The two bipartition classes differ in size by at most 1.
\end{theorem}

\begin{proof}
 Denote $w(A)$ the weight function defined in the proof of \cref{thm:bipartite}. We claim
  that
  \begin{displaymath}
      \left\lvert \sum_A w(A) \right\rvert \le 1.
  \end{displaymath}
  The sum $\sum_A w(A)$ over all arborescences $A$ expresses the difference
  between the number of ``positive'' and ``negative'' arborescences. 

  This sum
  can be calculated by the weighted matrix-tree
  theorem (see for example \cite{chaiken78Matrix}), which says the following:  Let $L=(L_{ij})$ be the weighted Laplace matrix,
  which has entry $L_{ij} = -w_{ij}$ for each arc $i\to j$, and $L_{ij}=0$
  for all off-diagonal entries that don't correspond to an arc.  The diagonal
  entries $L_{ii}$ are chosen to make the column sums zero.  Let $\tilde L$ be
  the matrix $L$ after removing the row and column corresponding to the root.
Then  
\begin{displaymath}
    \sum_A w(A) = \det \tilde L.
\end{displaymath}
In our case, the matrix $L$ looks as follows:
For a vertex with two incoming arcs,
 the corresponding column of $L$
has two entries $+1$ and $-1$, and hence the diagonal entry in this column, which is
supposed to balance the column sum, is zero.
For a vertex with a single incoming arc,
 the corresponding column of $L$
has an entry $+1$, and hence the diagonal entry is $-1$.

The matrix $\tilde L$ has therefore the following properties.
\begin{itemize}
\item All entries are $0$, $+1$, or $-1$.
    \item There are at most two nonzeros per column.
    \item If a column has two nonzeros, they are of opposite sign.
    \end{itemize}
The matrices arising from network flow problems have the same properties.    
The Theorem of Heller and Tompkins~\cite{HT-56}
characterises
the totally unimodular matrices among the matrices with the first two
properties, and it implies in this case that
$\tilde L$ is totally unimodular, and in particular, it has determinant
$0$ or $\pm1$.

It is easy to see this directly:
If there is a zero column, the determinant is zero.
If there is a column with a single nonzero, we expand the determinant by this column, and we obtain a smaller matrix that fulfills the same conditions.
Repeating this process, we either find a zero column at some point,
or we reduce the matrix to a trivial matrix of size  $1\times1$,
or to a matrix where each column contains a $+1$ and a $-1$. 
In the first and third cases,
$\det \tilde L=0$;
in the second case,
$\det \tilde L= \pm 1$.
\end{proof}

The procedure can easily be carried out combinatorially. Deleting the row of the root means
that all neighbours of the root lose an incoming arc.
A vertex $j$ with a single incoming arc corresponds to a
 column $j$ with a single nonzero entry, say $L_{ij}$.
 Expanding this column
 deletes
row $i$ and column $j$. In the graph, this corresponds to removing all edges out of $i$, which
means that the outneighbours of $i$ will have their indegree reduced, and the process can continue.

If the process stops before dismantling the whole graph, the determinant is zero, otherwise it is $\pm1$.

\subsection{Arborescences of degree one in the flip graph}

Another property preventing the existence of Hamiltonian cycles in the flip
graph is the existence of degree-one vertices, which might occur even when the
flip graph is not bipartite (and thus not a path). For example, the graph drawn on
\cref{fig:flip_G1} is an oriented graph with this property.
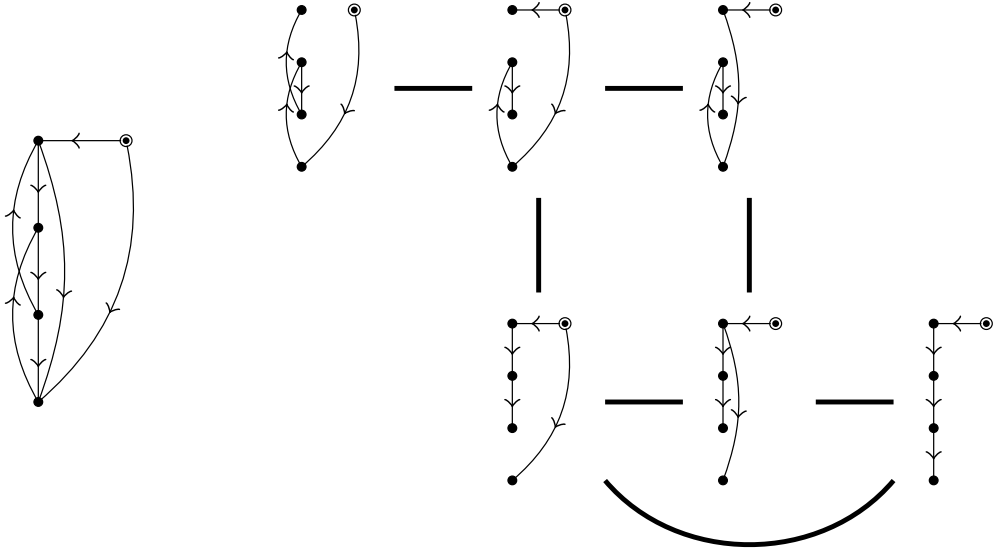
\begin{figure}[ht!]
  \centering
  \begin{tikzpicture}[decoration={
      markings,
      mark=at position 0.6 with {\arrow[scale = 1.5]{>}}}
    ]
    \begin{scope}[shift={(-3,-1.5)}]
      \foreach \i in {0,1,2,3}{
        \coordinate (v\i) at ($(0,{-\i})$); 
      }

      \node[vertex,double] (r) at (1,0){};
      \draw[postaction={decorate}]  (r) -- (v0);
      \draw[postaction={decorate}, bend left] (r) to (v3);
      
      \draw[postaction={decorate}]  (v0) -- (v1);
      \draw[postaction={decorate}] (v1) -- (v2);
      \draw[postaction={decorate}]  (v2) -- (v3);

      \draw[postaction={decorate}, bend left]  (v2) to (v0);
      \draw[postaction={decorate}, bend left]  (v3) to (v1);
      \draw[postaction={decorate}, bend left=20]  (v0) to (v3);

      \foreach \i in {0,...,3}{
        \node[vertex] at (v\i){};
      }
    \end{scope}

    \begin{scope}[scale =.6, shift={(0,0)}]
      \node[rectangle, minimum width = 1.5cm, minimum height = 2.5cm] (A) at
      (.5,-1.5) {};
      \foreach \i in {0,1,2,3}{
        \coordinate (v\i) at ($(0,{-\i})$); 
      }
      \node[vertex,double] (r) at (1,0){};
      \draw[postaction={decorate}, bend left] (r) to (v3);
      
      \draw[postaction={decorate}] (v1) -- (v2);

      \draw[postaction={decorate}, bend left]  (v2) to (v0);
      \draw[postaction={decorate}, bend left]  (v3) to (v1);
      
      \foreach \i in {0,...,3}{
        \node[vertex] at (v\i){};
      }
    \end{scope}

    \begin{scope}[scale =.6, shift={(4,0)}]
      \node[rectangle, minimum width = 1.5cm, minimum height = 2.5cm] (B) at
      (.5,-1.5) {};

      \foreach \i in {0,1,2,3}{
        \coordinate (v\i) at ($(0,{-\i})$); 
      }

      \node[vertex,double] (r) at (1,0){};
      \draw[postaction={decorate}]  (r) -- (v0);
      \draw[postaction={decorate}, bend left] (r) to (v3);
      
      \draw[postaction={decorate}] (v1) -- (v2);

      \draw[postaction={decorate}, bend left]  (v3) to (v1);

      \foreach \i in {0,...,3}{
        \node[vertex] at (v\i){};
      }
    \end{scope}

    \begin{scope}[scale =.6, shift={(4,-6)}]
      \node[rectangle, minimum width = 1.5cm, minimum height = 2.5cm] (C) at
      (.5,-1.5) {};
      
      \foreach \i in {0,1,2,3}{
        \coordinate (v\i) at ($(0,{-\i})$); 
      }

      \node[vertex,double] (r) at (1,0){};
      \draw[postaction={decorate}]  (r) -- (v0);
      \draw[postaction={decorate}, bend left] (r) to (v3);

      \draw[postaction={decorate}]  (v0) -- (v1);
      \draw[postaction={decorate}] (v1) -- (v2);

      \foreach \i in {0,...,3}{
        \node[vertex] at (v\i){};
      }
    \end{scope}

    \begin{scope}[scale =.6, shift={(8,0)}]
      \node[rectangle, minimum width = 1.5cm, minimum height = 2.5cm] (D) at
      (.5,-1.5) {};

      \foreach \i in {0,1,2,3}{
        \coordinate (v\i) at ($(0,{-\i})$); 
      }

      \node[vertex,double] (r) at (1,0){};
      \draw[postaction={decorate}]  (r) -- (v0);
      
      \draw[postaction={decorate}] (v1) -- (v2);

      \draw[postaction={decorate}, bend left]  (v3) to (v1);
      \draw[postaction={decorate}, bend left=20]  (v0) to (v3);

      \foreach \i in {0,...,3}{
        \node[vertex] at (v\i){};
      }
    \end{scope}

    \begin{scope}[scale =.6, shift={(8,-6)}]
      \node[rectangle, minimum width = 1.5cm, minimum height = 2.5cm] (E) at
      (.5,-1.5) {};

      \foreach \i in {0,1,2,3}{
        \coordinate (v\i) at ($(0,{-\i})$); 
      }

      \node[vertex,double] (r) at (1,0){};
      \draw[postaction={decorate}]  (r) -- (v0);
      
      \draw[postaction={decorate}]  (v0) -- (v1);
      \draw[postaction={decorate}] (v1) -- (v2);

      \draw[postaction={decorate}, bend left=20]  (v0) to (v3);
      
      \foreach \i in {0,...,3}{
        \node[vertex] at (v\i){};
      }
    \end{scope}

    \begin{scope}[scale =.6, shift={(12,-6)}]
      \node[rectangle, minimum width = 1.5cm, minimum height = 2.5cm] (F) at
      (.5,-1.5) {};

      \foreach \i in {0,1,2,3}{
        \coordinate (v\i) at ($(0,{-\i})$); 
      }

      \node[vertex,double] (r) at (1,0){};
      \draw[postaction={decorate}]  (r) -- (v0);
      
      \draw[postaction={decorate}]  (v0) -- (v1);
      \draw[postaction={decorate}] (v1) -- (v2);
      \draw[postaction={decorate}]  (v2) -- (v3);

      \foreach \i in {0,...,3}{
        \node[vertex] at (v\i){};
      }
    \end{scope}

    \draw[ultra thick] (A) -- (B) -- (C) -- (E) -- (D) (B) -- (D);
    \draw[ultra thick] (E) -- (F) edge[ultra thick, bend left = 50] (C);
  \end{tikzpicture}
  \caption{An oriented graph and its flip graph, which contains a degree-one vertex}
  \label{fig:flip_G1}
\end{figure}

The degree-one vertices in the flip graph can be characterised as follows:
\begin{observation}
  Let $\vec{G}$ be a directed graph rooted in $r$ and $A$ an arborescence of
  $\vec{G}$ such that $A$ has degree one in the flip graph $\FD{\vec{G}}r$, by
  flipping in some arc $uv$. Then $\vec{G}-uv$ is built on $A$.
\end{observation}
\begin{proof}
  Let $xy \in E(\vec{G}-uv) \setminus E(A)$. Assume towards contradiction that
  $x \notin D_A(y)$. So in $A$, the arc $xy$ can be flipped in as $x$ is
  accessible from $r$ without passing through $y$ in $A$. Thus $xy = uv$, a
  contradiction.
\end{proof}

\section{A Gray code for arborescences in tournaments}\label{sec:tournaments}

We prove the following refinement of \cref{thm:tournaments}:
\begin{theorem}\label{thm:directed}
  Let $\vec{G}$ be a directed multi-graph and $r$ be a vertex of $\vec{G}$ such that
  the support of $\vec{G}-r$ is a clique (in particular, $\vec G$ can have arcs between two vertices $u$ and $v$, possibly in opposite directions). Then the flip graph $\FD{\vec{G}}{r}$ is empty or
  has a Hamiltonian path.
\end{theorem}

First, we observe that \cref{thm:directed} is optimal in the sense that there exist
tournaments with a flip graph that contains a Hamiltonian path but no Hamiltonian
cycle. The graph $\vec G$ drawn on \cref{fig:flip_bipartite7} with its flip
graph is an example of such tournament, and note also that $\vec G$ falls into
the conditions of \cref{thm:bipartite}.
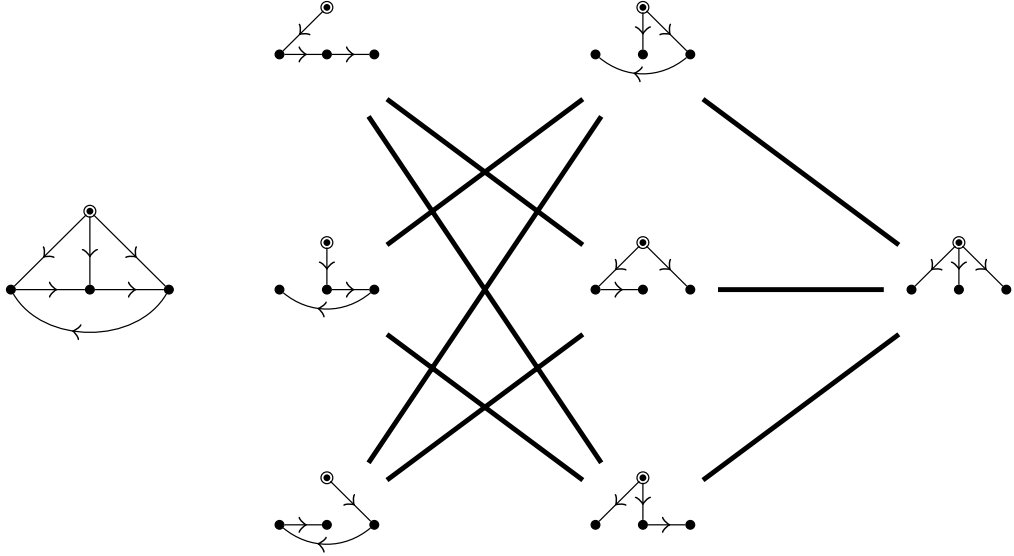
\begin{figure}[h!]
  \centering
  \begin{tikzpicture}[scale=.9,decoration={
      markings,
      mark=at position 0.6 with {\arrow[scale=1.5]{>}}}
    ]

    \begin{scope}[shift={(-3,-3)}]
      \node[vertex,double] at (0,1) (r) {};
      \node[vertex] at (1,0) (x) {};
      \node[vertex] at (0,0) (v) {};
      \node[vertex] at (-1,0) (u) {};

      \draw[postaction={decorate}]  (r) -- (u);    
      \draw[postaction={decorate}]  (r) -- (v);    
      \draw[postaction={decorate}]  (r) -- (x);    
      \draw[postaction={decorate}]  (u) -- (v);    
      \draw[postaction={decorate}]  (v) -- (x);
      \draw[postaction={decorate}, bend left = 60]  (x) to (u);
    \end{scope}

    \foreach \i/\x/\y in {
      1/0/0,2/4/0,
      3/0/-3,4/4/-3,5/8/-3,
      6/0/-6,7/4/-6}
    {
      \begin{scope}[shift={(\x,\y)}, scale=.6]
        \node[anchor=center,circle,inner sep=.6cm] (A\i) at (0,0) {};
        \node[vertex,double] at (0,1) (r\i) {};
        \node[vertex] at (1,0) (x\i) {};
        \node[vertex] at (0,0) (v\i) {};
        \node[vertex] at (-1,0) (u\i) {};
      \end{scope}
    }

    \draw[ultra thick] (A4) -- (A1) -- (A7) -- (A3) -- (A2) -- (A6) -- (A4); 
    \draw[ultra thick] (A2) -- (A5);
    \draw[ultra thick] (A4) -- (A5);
    \draw[ultra thick] (A7) -- (A5);
    
    \draw[postaction={decorate}]  (r1) -- (u1);    
    \draw[postaction={decorate}]  (u1) -- (v1);    
    \draw[postaction={decorate}]  (v1) -- (x1);
    
    \draw[postaction={decorate}]  (r2) -- (v2);
    \draw[postaction={decorate}]  (r2) -- (x2);
    \draw[postaction={decorate}, bend left=40]  (x2) to (u2);

    \draw[postaction={decorate}]  (r3) -- (v3);
    \draw[postaction={decorate}]  (v3) -- (x3);
    \draw[postaction={decorate}, bend left=40]  (x3) to (u3);

    \draw[postaction={decorate}]  (r4) -- (u4);
    \draw[postaction={decorate}]  (u4) -- (v4);
    \draw[postaction={decorate}]  (r4) -- (x4);

    \draw[postaction={decorate}]  (r5) -- (u5);
    \draw[postaction={decorate}]  (r5) -- (v5);
    \draw[postaction={decorate}]  (r5) -- (x5);

    \draw[postaction={decorate}]  (r6) -- (x6);
    \draw[postaction={decorate}, bend left=40]  (x6) to (u6);
    \draw[postaction={decorate}]  (u6) -- (v6);

    \draw[postaction={decorate}]  (r7) -- (u7);
    \draw[postaction={decorate}]  (r7) -- (v7);
    \draw[postaction={decorate}]  (v7) -- (x7);

  \end{tikzpicture}
  \caption{An oriented graph with seven arborescences and whose flip graph is
  bipartite, so without Hamiltonian cycle.}\label{fig:flip_bipartite7}
\end{figure}

The rest of this section consists in the proof of \cref{thm:directed}. We first observe that by \cref{lem:arc_multiplication}, it suffices to prove \cref{thm:directed} for directed graphs without multiple arcs.
Let $\vec{\mathcal{G}}$ denote the set of pairs $(\vec{G},r)$ where $\vec G$ is a directed graph, $r$ is a root in $\vec G$ and the support of $\vec{G}-r$ is a clique. Note that $\vec{\mathcal{G}}$ is stable under the following
operations:
\begin{itemize}
\item deletion of arcs incident to $r$:
  $\forall u, (\vec{G}-ru,r) \in \vec{\mathcal{G}}$
\item contraction of arcs $ru$ incident to $r$: let $\vec{G'}$ be the graph
  obtained from $\vec{G}$ by removing $r$ and $u$ and replacing them by a
  vertex $w$ with arcs $wx$ for all $x \in N^+(r) \cup N^+(u)$. Then
  $(\vec{G'},w) \in \vec{\mathcal{G}}$
\end{itemize}

We proceed by induction. Let $(\vec{G},r) \in \vec{\mathcal{G}}$. The arcs
ending in $r$ do not belong to any arborescence, so by \cref{obs:arc_deletion}, we assume without loss of
generality that $N^-(r) = \emptyset$. If $r$ has
only one outneighbour $u$, then the arc $ru$ is contained in every arborescence
of $\vec{G}$ rooted in $r$. Hence, by \cref{lem:arc_contraction}, $\FD{\vec{G}}{r}$ is
isomorphic to $\FD{\vec{G}-r}{u}$ and by induction both admit a Hamiltonian path
or are empty.

If $r$ has at least two outneighbours $u$ and $v$, then they are connected by an
arc, say $u \to v$, because the support of $\vec{G}-r$ is a clique. The vertices $u$ and
$v$ might also be connected by the arc $v \to u$, but we do not care at this
stage of the proof. We choose $u$ and $v$ as follows: among all pairs
$(u,v) \in N^+(r)^2$ with $uv \in \vec G$, we choose one such that $u$ has an
outneighbourhood of maximal size.

The set of arborescences of $\vec{G}$ can be partitioned into four \defin{types}
of arborescences: those that do not contain the arc $e=ru$, those that contain
$ru$ and $f=rv$, those that contain $ru$ and $g=uv$ and those that $ru$ but
neither $rv$ nor $rv$. These types partition $\FD{\vec{G}}{r}$ into four
subsets, that we will denote $\T_{-e}$, $\T_{/e/f}$, $\T_{/e/g}$ and
$\T_{-f-g/e}$ respectively.

\subsection{Structure of each type}
\begin{lemma}\label{lem:structure_types}
  If the flip graph of $\vec G$ rooted in $r$ is non-empty, the different types
  of arborescences induce the following structures in $\FD{\vec{G}}{r}$:
  \begin{itemize}
  \item $\FD{\vec{G}}{r}[\T_{-e}]$ is empty or contains a Hamiltonian path $P_{-e}$,
  \item $\FD{\vec{G}}{r}[\T_{-f-g/e}]$ is empty or contains a Hamiltonian path $P_{-f-g/e}$, 
  \item $\FD{\vec{G}}{r}[\T_{/e/f} \cup \T_{/e/g}]$ contains a spanning ladder.
  \end{itemize}
\end{lemma}
\begin{proof}
  As a direct application of \cref{obs:arc_deletion}, we get that
  $\FD{\vec{G}}{r}[\T_{-e}]$ is isomorphic to $\FD{\vec{G}-e}{r}$, so by
  induction $\FD{\vec{G}}{r}[\T_{-e}]$ is empty or also admits a Hamiltonian
  path $P_{-e}$.

  Using \cref{obs:arc_deletion,lem:arc_contraction}, we get that
  $\FD{\vec{G}}{r}[\T_{-f-g/e}]$ is isomorphic to $\FD{\vec{H}}{w}$, where
  $\vec{H}$ is obtained from $\vec{G}$ by contracting $e$ into $w$ and deleting
  $wv$. By induction, $\FD{\vec{G}}{r}[\T_{-f-g/e}]$ is empty or admits a
  Hamiltonian path $P_{-f-g/e}$.

  Finally, note that $\FD{\vec{G}}{r}[\T_{/e/f}]$ and
  $\FD{\vec{G}}{r}[\T_{/e/g}]$ are isomorphic via flipping the arc $f$ into the
  arc $g$. They are also isomorphic to $\FD{\vec{H}}{w}$, where $H$ is obtained
  from $\vec{G}$ by contracting $e$ and $g$ into $w$. Moreover, note that
  $\FD{\vec{G}}{r}[\T_{/e/f} \cup \T_{/e/g}]$ is not empty, because $H$ admits an
  arborescence if and only $G$ does, by \cref{lem:completion} applied to the subtree of $G$ composed of the edges $e$ and $g$. Combining all these arguments with
  \cref{lem:arc_contraction}, we deduce that
  $\FD{\vec{G}}{r}[\T_{/e/f} \cup \T_{/e/g}]$ contains a spanning ladder (see
  \cref{fig:types}).

\begin{figure}[h!]
  \centering
  \begin{tikzpicture}
    \coordinate (a1) at (0,0);
    \coordinate (a2) at (-.5,1);
    \coordinate (a3) at (.5,3);
    \coordinate (a4) at (0,4);
    \path[draw,use Hobby shortcut] (a1) .. (a2) .. (a3) .. (a4);
    \node[draw, circle, inner sep=1pt, fill=black] at (a1) {};
    \node[draw, circle, inner sep=1pt, fill=black] at (a4) {};
    
    \foreach \i in {0,...,8}{
      \node[draw, circle, inner sep=1pt, fill=black] (ll\i) at ($(3,{\i*.5})$) {}; 
      \node[draw, circle, inner sep=1pt, fill=black] (lr\i) at ($(3.5,{\i*.5})$) {}; 
      \draw (ll\i) -- (lr\i);
    }
    \draw (ll0) -- (ll8) (lr0) -- (lr8);

    \coordinate (b1) at (6.5,0);
    \coordinate (b2) at (6,1);
    \coordinate (b3) at (7,3);
    \coordinate (b4) at (6.5,4);
    \path[draw,use Hobby shortcut] (b1) .. (b2) .. (b3) .. (b4);
    \node[draw, circle, inner sep=1pt, fill=black] at (b1) {};
    \node[draw, circle, inner sep=1pt, fill=black] at (b4) {};

    \node[below=.2cm] at (a1) {$P_{-e} \subset \FD{\vec{G}}{r}[\T_{-e}]$};
    \node[below=.2cm] at (b1) {$P_{-f-g/e} \subset \FD{\vec{G}}{r}[\T_{-f-g/e}]$};
    \node[above=.2cm] at (3.25,4) {$\FD{\vec{G}}{r}[\T_{/e/f}\cup \T_{/e/g}]$};
  \end{tikzpicture}
  \caption{Spanning structures within the different types of arborescences.}
  \label{fig:types}
\end{figure}
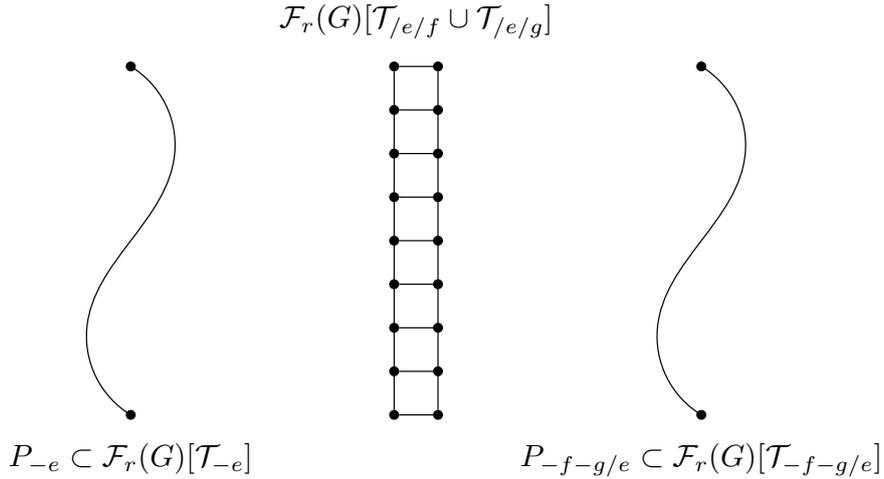
\end{proof}

\subsection{Assembling the pieces}

\begin{proof}[Proof of \cref{thm:directed}]
  We first argue that one can assume without loss of generality that
  $\T_{-f-g/e}$ and $\T_{-e}$ are non-empty. If both are empty, then by
  \cref{lem:structure_types}
  $\FD{\vec G}r = \FD{\vec{G}}{r}[\T_{/e/f} \cup \T_{/e/g}]$ contains a spanning
  ladder and \cref{thm:directed} follows directly from \cref{lem:ladder}. If
  $\T_{-f-g/e}$ is empty and $\T_{-e}$ is not, then $\FD{\vec{G}}{r}[\T_{-e}]$
  contains a Hamiltonian path $P_{-e}$ by \cref{lem:structure_types}. By
  flipping in $e$ in one of the extremities of $P_{-e}$, one obtains an
  arborescence in $\T_{/e/f} \cup \T_{/e/g}$. By \cref{lem:structure_types}
  $\FD{\vec{G}}{r}[\T_{/e/f} \cup \T_{/e/g}]$ contains a ladder, and hence
  contains a Hamiltonian cycle, which can be used to extend $P_{-e}$ into a
  Hamiltonian path of $\FD{\vec G}r$. Similarly, we can assume that $\T_{-e}$ is
  non-empty.

  For $A$ in $\T_{-f-g/e}$, let $A'$ and $A''$ be the arborescences in
  $\T_{/e/f}$ (respectively $\T_{/e/g}$) obtained from $A$ by flipping in $f$
  (respectively $g$). 

  Assume for now that the extremities $A_1$ and $A_2$ of $P_{-f-g/e}$ have
  distinct $A_i'$. By \cref{lem:structure_types}, $\FD{\vec{G}}{r}[\T_{/e/f} \cup \T_{/e/g}]$ contains a ladder, so by \cref{lem:ladder},
  it contains a
  Hamiltonian path going from $A_1'$ to $A_2'$ or $A_2''$. Along with
  $P_{-f-g/e}$ (obtained from \cref{lem:structure_types}), one obtains a
  Hamiltonian cycle of
  $\FD{\vec{G}}{r}[\T_{/e/f} \cup \T_{/e/g} \cup \T_{-f-g/e}]$. Finally, all
  arborescences in $\T_{-e}$ are adjacent to some arborescence in
  $\T_{/e/f} \cup \T_{/e/g} \cup \T_{-f-g/e}$ by flipping in the edge $e$, so
  $\FD{\vec{G}}{r}$ contains a Hamiltonian path starting in $\T_{-e}$ (see
  \cref{fig:easy_case}).

  \begin{figure}[h!]
    \centering
    \begin{tikzpicture}
      \coordinate (a1) at (0,0);
      \coordinate (a2) at (-.5,1);
      \coordinate (a3) at (.5,3);
      \coordinate (a4) at (0,4);
      \path[draw, very thick, red, use Hobby shortcut] (a1) .. (a2) .. (a3) .. (a4);
      
      \foreach \i in {0,...,8}{
        \coordinate (ll\i) at ($(3,{\i*.5})$); 
        \coordinate (lr\i) at ($(3.5,{\i*.5})$); 
        \draw (ll\i) -- (lr\i);
      }
      \draw (ll0) -- (ll8) (lr0) -- (lr8);

      \coordinate (b1) at (6.5,0);
      \coordinate (b2) at (6,1);
      \coordinate (b3) at (7,3);
      \coordinate (b4) at (6.5,4);
      \path[draw, very thick, red, use Hobby shortcut] (b1) .. (b2) .. (b3) .. (b4);
      
      \draw[very thick, red] (b1) -- (lr1);
      \draw[very thick, red] (b4) -- (lr6);
      \draw[very thick, red] (lr1) -- (lr0) -- (ll0) -- (ll2) -- (lr2) -- (lr3)
      -- (ll3) -- (ll4) -- (lr4) -- (lr5) -- (ll5) -- (ll8) -- (lr8) -- (lr6); 
      \draw[very thick, red] (a1) -- (1.5,.5);
      
      \node[draw, circle, inner sep=1pt, fill=black] at (a1) {};
      \node[draw, circle, inner sep=1pt, fill=black] at (a4) {};
      \node[draw, circle, inner sep=1pt, fill=black] at (b1) {};
      \node[draw, circle, inner sep=1pt, fill=black] at (b4) {};
      \foreach \i in {0,...,8}{
        \node[vertex] at (ll\i){};
        \node[vertex] at (lr\i){};
      }
      \node[right] at (b1) {$A_1$};
      \node[right] at (b4) {$A_2$};
      
      \node[below=.2cm] at (a1) {$P_{-e} \subset \FD{\vec{G}}{r}[\T_{-e}]$};
      \node[below=.2cm] at (b1) {$P_{-f-g/e} \subset \FD{\vec{G}}{r}[\T_{-f-g/e}]$};
      \node[above=.2cm] at (3.25,4) {$\FD{\vec{G}}{r}[\T_{/e/f}\cup \T_{/e/g}]$};
    \end{tikzpicture}
    \caption{Hamiltonian path if $A_1' \neq A_2'$.}
    \label{fig:easy_case}
  \end{figure}
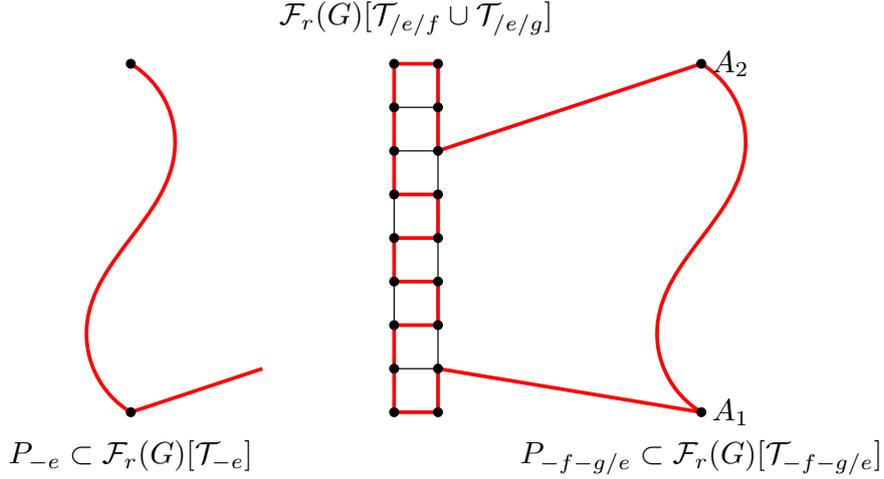

  Thus we can assume without loss of generality that the extremities $A_1$ and
  $A_2$ of $P_{-f-g/e}$ have $A_1'= A_2'$ (and $A_1'' = A_2 ''$). As $A_i'$ is obtained from $A_i$ by flipping in the arc $f$, this implies that $A_1$ and
  $A_2$ only differ on one arc entering $v$, so the path $P_{-f-g/e}$ from
  \cref{lem:structure_types} is in fact a Hamiltonian cycle of
  $\FD{\vec{G}}{r}[\T_{-f-g/e}]$. Either $A' = A_1'$ for every arborescence $A$ in
  $\T_{-f-g/e}$ or there are two arborescences $A$ and $B$, consecutive on
  $P_{-f-g/e}$, such that $A' \neq B'$. In the second case,
  $\FD{\vec{G}}{r}[\T_{-f-g/e}]$ contains a Hamiltonian path whose extremities
  are $A$ and $B$, so we are back in the case of the previous paragraph.

  Thus we can assume that for every $A \in \T_{-f-g/e}$, we have $A' = A_1'$ and
  $A''=A_1''$. We now characterise the structure of tournaments for which this is possible:

\begin{lemma}\label{lem:flip_clique}
  Let $\vec{G}$ be a directed graph rooted in a vertex $w$, such that the support of
  $\vec{G} -w$ is a clique. Let $v \in V(\vec{G})$ such that
  flipping in $wv$ in any arborescence rooted in $w$ that does not contain $wv$,
  results in the same arborescence $A$. Then the following hold:
  \begin{enumerate}[label=(\arabic*),ref=(\arabic*)]
  \item\label{it:flip_clique1} $\FD{\vec{G} -wv}w$ is a clique,
  \item\label{it:flip_clique2} all vertices other $v$ have a single path from $w$ to them in
    $\vec G - wv$,
  \item\label{it:flip_clique3} all paths going from $w$ to $v$ in $\vec G - wv$ have length at least 2,
  \item\label{it:flip_clique4} $\vec G -wv$ has the following structure:
    \begin{enumerate}[label=(\alph*),ref=(4\alph*)]
    \item\label{it:flip_clique4a} If $\FD{\vec{G}-wv}{w}$ is a single vertex,
      then $\vec G - wv$ is built on the directed path $w, v_1, \dots,v_n$, with
      $v = v_k$ for some fixed $k>1$. In other words, $\vec G$ contains the
      graph $L_{k,n}$ (drawn on \cref{fig:flip_Lkn} in black) as a subgraph, and
      the arcs in $E(G) \setminus E(L_{k,n})$ are either ending at $w$ or of the
      form $v_{i+1} \to v_i$ (drawn on \cref{fig:flip_Lkn} in red).
    \item\label{it:flip_clique4b} If $\FD{\vec{G}-wv}{w}$ is a clique on at
      least two vertices, then $\vec{G} -wv$ is built on the directed path
      $w,v_1, \dots,v_n$ with an additional vertex $v$ and all arcs $v_iv$ with
      $i\ge k-1$ for some fixed $1 < k \le n$. In other words, $\vec G$ contains
      the graph $M_{k,n}$ (drawn on \cref{fig:flip_Mkn} in black) as a subgraph, and the
      arcs in $E(G) \setminus E(M_{k,n})$ are either $v \to v_{k-1}$, or ending
      at $w$, or of the form $v_{i+1} \to v_i$ (drawn on \cref{fig:flip_Mkn} in red).
    \end{enumerate}
  \end{enumerate}
\end{lemma}

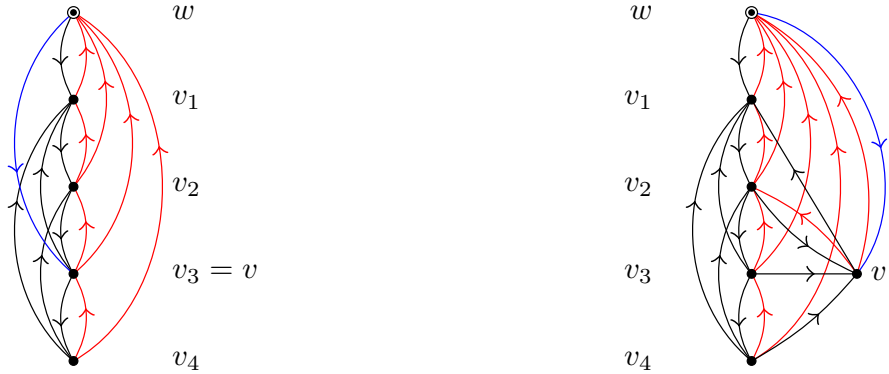
\begin{figure}[h!]
  \centering
  \begin{subfigure}{.4\textwidth}
    \centering
    \begin{tikzpicture}[decoration={
        markings,
        mark=at position 0.6 with {\arrow[scale=1.5]{>}}}
      ]
      
      \foreach \i in {1,...,5}{
        \coordinate (v\i) at ($(0,-\i)$); 
      }

      \foreach \i in {1,...,4}{
        \pgfmathsetmacro{\j}{\i+1}
        \draw[postaction={decorate}, bend right=30,red]  (v\j) to (v\i);
        \draw[postaction={decorate}, bend right=30]  (v\i) to (v\j);
      }
      \draw[postaction={decorate}, bend right=60,red]  (v5) to (v1);
      \draw[postaction={decorate}, bend right=50,red]  (v4) to (v1);
      \draw[postaction={decorate}, bend right=40,red]  (v3) to (v1);

      \draw[postaction={decorate}, bend left=40]  (v5) to (v3);
      \draw[postaction={decorate}, bend left=50]  (v5) to (v2);      
      \draw[postaction={decorate}, bend left=40]  (v4) to (v2);

      \draw[postaction={decorate}, bend right=50, blue]  (v1) to (v4);

      \foreach \i in {2,...,5}{
        \node[vertex] at (v\i){};
      }
      \node[vertex,double] at (v1){};
      
      \node[right = 1cm of v1] {$w$};
      \node[right = 1cm of v2] {$v_1$};
      \node[right = 1cm of v3] {$v_2$};
      \node[right = 1cm of v4] {$v_3=v$};
      \node[right = 1cm of v5] {$v_4$};
    \end{tikzpicture}
    \caption{Illustration of \ref{it:flip_clique4a} with ${n=4}$ and $k=3$}\label{fig:flip_Lkn}
  \end{subfigure}
  \hfill
  \begin{subfigure}{.4\textwidth}
    \centering
    \begin{tikzpicture}[decoration={
        markings,
        mark=at position 0.6 with {\arrow[scale=1.5]{>}}}
      ]

        \foreach \i in {1,...,5}{
        \coordinate (v\i) at ($(0,-\i)$); 
      }

      \foreach \i in {1,...,4}{
        \pgfmathsetmacro{\j}{\i+1}
        \draw[postaction={decorate}, bend right=30,red]  (v\j) to (v\i);
        \draw[postaction={decorate}, bend right=30]  (v\i) to (v\j);
      }
      \draw[postaction={decorate}, bend right=60,red]  (v5) to (v1);
      \draw[postaction={decorate}, bend right=50,red]  (v4) to (v1);
      \draw[postaction={decorate}, bend right=40,red]  (v3) to (v1);

      \draw[postaction={decorate}, bend left=40]  (v5) to (v3);
      \draw[postaction={decorate}, bend left=50]  (v5) to (v2);      
      \draw[postaction={decorate}, bend left=40]  (v4) to (v2);

      \coordinate (v) at ($(v4)+(1.2,0)$); 

      \draw[postaction={decorate}, bend left=60, blue]  (v1) to (v);

      \draw[postaction={decorate}, bend right=15]  (v3) to (v);
      \draw[postaction={decorate}, bend right=15,red]  (v) to (v3);
      \draw[postaction={decorate}]  (v4) to (v);
      \draw[postaction={decorate}, bend right=10]  (v5) to (v);
      \draw[postaction={decorate}]  (v) to (v2);
      \draw[postaction={decorate}, bend right=40,red]  (v) to (v1);

      \foreach \i in {2,...,5}{
        \node[vertex] at (v\i){};
      }
      \node[vertex] at (v) {};
      \node[vertex,double] at (v1) {};

      \node[left = 1cm of v1] {$w$};
      \node[left = 1cm of v2] {$v_1$};
      \node[left = 1cm of v3] {$v_2$};
      \node[left = 1cm of v4] {$v_3$};
      \node[left = 1cm of v5] {$v_4$};
      
      \node[right = 1pt of v] {$v$};
    \end{tikzpicture}

    \caption{Illustration of \ref{it:flip_clique4b} with ${n=4}$ and $k=3$}\label{fig:flip_Mkn}
  \end{subfigure}
  \caption{The structure of tournaments $\vec G$, such that all arborescences of $\vec G
    -wv$ rooted in $w$ produce the same tree when flipping in $wv$. The arc $wv$
  is drawn in blue, the black arcs are present in $\vec G$ while the red arcs
  are optional.}
  \label{fig:flip_clique}
\end{figure}

\begin{proof}[Proof of \cref{lem:flip_clique}]
  All arborescences of $\vec G -wv$ rooted in $w$ have in common all arcs but
  the one entering $v$ because they result in the same arborescence when flipping in
  $wv$. Thus $\FD{\vec G -wv }w$ is a clique, which proves \ref{it:flip_clique1}.
  
  Perform a BFS from $w$ in $\vec G -wv$. If each layer of the BFS contains only
  one vertex, then $\vec{G}-{wv}$ consists of a directed path $w,v_1, \dots, v_n$
  and some backedges, in other words $\vec G-wv$ is built on a directed path. In
  fact, $\vec G$ contains all arcs $v_iv_j$ with $i>j+1$ because $\vec G
  -w$ has complete support (see \cref{fig:flip_Lkn}). The additional possible
  backedges are either ending at $w$, or arcs of the form $v_{i+1} \to v_i$,
  which proves \ref{it:flip_clique4a}. Moreover, note that $\vec G -wv$
  has only one arborescence rooted in $w$.
  
  If some layer of the BFS contains more than one vertex, assume towards
  contradiction that some vertex $x$ different from $v$ has two inneighbours
  with depth no larger than that of $x$ (i.e. either in the same layer or the
  previous one). Thus, there are at least two paths from $w$ to $x$ in
  $\vec G -wv$. By \cref{lem:completion}, each of these paths can be completed
  into an arborescence of $\vec G -wv$ rooted in $w$. These two arborescences
  have different arc entering $x$, hence flipping in $wv$ results in different
  arborescences, a contradiction. So all layers are either reduced to one single
  vertex, or contain two vertices $v_k$ and $v$ and $v_k \notin N^{+}_G(v)$. Hence all
  vertices but $v$ have at most one inneighbour at smaller depth in the BFS. Let
  $k$ be the depth of $v$. All other vertices at depth at least $k$ are
  inneighbours of $v$ because $\vec G$ has complete support. Hence $\vec G -wv$
  contains as a spanning subgraph the path $w,v_1 , \dots,v_n$ with one
  additional vertex $v$ such that $v_i\to v$ if $i \ge k-1$. In addition, since
  $\vec G - w$ has complete support, $\vec G-w$ also contains all backedges of
  the form $v_j \to v_i$ with $i > j+1$, or $v \to v_i$ with $i < k-1$. In other
  words, the graph $M_{k,n}$ drawn in black on \cref{fig:flip_Mkn} is a subgraph of
  $G-wv$. Finally, $\vec G$ can also contain other backedges, namely
  $v \to v_{k-1}$, the arcs ending at $w$, or the arcs of the form
  $v_{i+1} \to v_i$, which proves \ref{it:flip_clique4b}.

  Note that in both cases, $v \neq v_1$ because $wv \notin \vec G -wv$, so
  $k > 1$, which implies \ref{it:flip_clique3}. Moreover, as backedges by
  definition never belong to a path from the root to a vertex, all vertices but
  $v$ have only one path from $w$ to them in $\vec G -wv$, which proves
  \ref{it:flip_clique2}.
\end{proof}

  By \cref{lem:flip_clique} applied to $\vec G/e$ with $w$ being
  the contraction of $ru$ and $v:= v$, this implies that $\vec G/e$
  has the structure described in \ref{it:flip_clique4} for some $k$ and $n$ and
  that $\FD{\vec{G}}{r}[\T_{-f-g/e}]$ is a clique (by \ref{it:flip_clique1}).
  
  For every $B \in \FD{\vec{G}}{r}[\T_{-e}]$, let $\tilde B$ be the arborescence of
  $\FD{\vec{G}}{r}[\T_{/e/f} \cup \T_{/e/g} \cup \T_{-f-g/e}]$ obtained by
  flipping in $e$. Let $B_1$ and $B_2$ be the extremities of $P_{-e}$. We now do a
  case analysis on the types of $\tilde B_1$ and $\tilde B_2$.

  \subsubsection{$\tilde B_1$ or $\tilde B_2$ belongs to
    $\FD{\vec{G}}{r}[\T_{-f-g/e}]$.} Then $\FD{\vec{G}}{r}$ contains a
  Hamiltonian path that first visits $\FD{\vec{G}}{r}[\T_{-e}]$ via $P_{-e}$,
  then all the vertices of $\FD{\vec{G}}{r}[\T_{-f-g/e}]$ and finally those of
  $\FD{\vec{G}}{r}[\T_{/e/f} \cup \T_{/e/g}]$ (see \cref{fig:case1}).

  \begin{figure}[h!]
    \centering
    \begin{tikzpicture}
      \coordinate (a1) at (0,0);
      \coordinate (a2) at (-.5,1);
      \coordinate (a3) at (.5,3);
      \coordinate (a4) at (0,4);
      \path[draw, very thick, red, use Hobby shortcut] (a1) .. (a2) .. (a3) .. (a4);
      
      \foreach \i in {0,...,8}{
        \coordinate (ll\i) at ($(3,{\i*.5})$); 
        \coordinate (lr\i) at ($(3.5,{\i*.5})$); 
        \draw (ll\i) -- (lr\i);
      }
      \draw (ll0) -- (ll8) (lr0) -- (lr8);

      \coordinate (b1) at (7,2);
      \coordinate (b2) at (6.5,2);
      \coordinate (b3) at (7,2.5);
      \coordinate (b4) at (6.5,2.5);
      \draw[very thick, red]  (b1) -- (b2) -- (b3) -- (b4);
      \draw (b3) -- (b1) -- (b4) -- (b2);

      \draw[very thick, red] (b4) -- (lr6);
      \draw[very thick, red] (lr6) -- (lr8) -- (ll8) -- (ll0) -- (lr0) -- (lr5);
      \draw[very thick, red] (a1) edge[bend right= 40] (b1);
      
      \node[draw, circle, inner sep=1pt, fill=black] at (a1) {};
      \node[draw, circle, inner sep=1pt, fill=black] at (a4) {};
      \node[draw, circle, inner sep=1pt, fill=black] at (b1) {};
      \node[draw, circle, inner sep=1pt, fill=black] at (b4) {};
      \foreach \i in {0,...,8}{
        \node[vertex] at (ll\i){};
        \node[vertex] at (lr\i){};
      }

      \node[left] at (a1) {$B_1$};
      \node[right] at (b1) {$\tilde B_1$};

      \node[below=.2cm] at (a1) {$P_{-e} \subset \FD{\vec{G}}{r}[\T_{-e}]$};
      \node[above=.2cm] at (b3) {$\FD{\vec{G}}{r}[\T_{-f-g/e}]$};
      \node[above=.2cm] at (3.25,4) {$\FD{\vec{G}}{r}[\T_{/e/f}\cup \T_{/e/g}]$};
    \end{tikzpicture}
    \caption{Hamiltonian path if $\tilde B_1 \in \FD{\vec{G}}{r}[\T_{-f-g/e}]$.}
    \label{fig:case1}
  \end{figure}
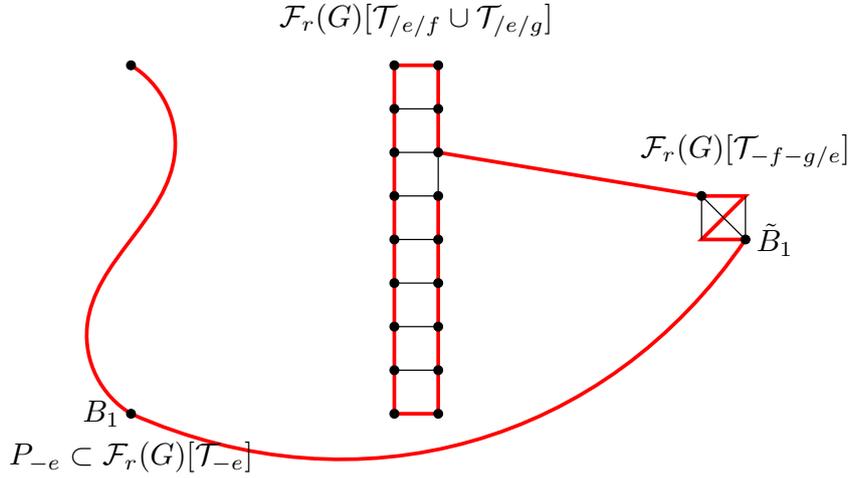

  \subsubsection{$\tilde B_1$ or $\tilde B_2$ belongs to
    $\FD{\vec{G}}{r}[\T_{/e/f} \cup \T_{/e/g}] \setminus
    \{A',A''\}$.}\label{subsubsec:1}

  Without loss of generality, assume that $\tilde B_1$ belongs to
  $\FD{\vec{G}}{r}[\T_{/e/f} \cup \T_{/e/g}] \setminus \{A',A''\}$. By
  \cref{lem:ladder}, $\FD{\vec{G}}{r}[\T_{/e/f} \cup \T_{/e/g}]$ contains a
  Hamiltonian path starting at $\tilde B_1$ and ending at $A'$ or $A''$. As a
  result, $\FD{\vec{G}}{r}$ contains a Hamiltonian path visiting first
  $P_{-e}$, then $\FD{\vec{G}}{r}[\T_{/e/f} \cup \T_{/e/g}]$ and finally
  $\FD{\vec{G}}{r}[\T_{-f-g/e}]$ (see \cref{fig:case2}).
  \begin{figure}[h!]
    \centering
    \begin{tikzpicture}
      \coordinate (a1) at (0,0);
      \coordinate (a2) at (-.5,1);
      \coordinate (a3) at (.5,3);
      \coordinate (a4) at (0,4);
      \path[draw, very thick, red, use Hobby shortcut] (a1) .. (a2) .. (a3) .. (a4);
      
      \foreach \i in {0,...,8}{
        \coordinate (ll\i) at ($(3,{\i*.5})$); 
        \coordinate (lr\i) at ($(3.5,{\i*.5})$); 
        \draw (ll\i) -- (lr\i);
      }
      \draw (ll0) -- (ll8) (lr0) -- (lr8);

      \coordinate (b1) at (7,2);
      \coordinate (b2) at (6.5,2);
      \coordinate (b3) at (7,2.5);
      \coordinate (b4) at (6.5,2.5);
      \draw[very thick, red]  (b1) -- (b2) -- (b3) -- (b4);
      \draw (b3) -- (b1) -- (b4) -- (b2);

      \draw[very thick, red] (b4) -- (lr6);
      \draw[very thick, red] (ll2) -- (ll0) -- (lr0) -- (lr3)
      -- (ll3) -- (ll4) -- (lr4) -- (lr5) -- (ll5) -- (ll8) -- (lr8) -- (lr6); 
      \draw[very thick, red] (a1) edge[bend right= 40] (ll2);
      
      \node[draw, circle, inner sep=1pt, fill=black] at (a1) {};
      \node[draw, circle, inner sep=1pt, fill=black] at (a4) {};
      \node[draw, circle, inner sep=1pt, fill=black] at (b1) {};
      \node[draw, circle, inner sep=1pt, fill=black] at (b4) {};
      \foreach \i in {0,...,8}{
        \node[vertex] at (ll\i){};
        \node[vertex] at (lr\i){};
      }

      \node[left] at (a1) {$B_1$};
      \node[left] at (ll2) {$\tilde B_1$};
      \node[left] at (ll6) {$A'$};
      \node[right] at (lr6) {$A''$};

      \node[below=.2cm] at (a1) {$P_{-e} \subset \FD{\vec{G}}{r}[\T_{-e}]$};
      \node[above=.2cm] at (b3) {$\FD{\vec{G}}{r}[\T_{-f-g/e}]$};
      \node[above=.2cm] at (3.25,4) {$\FD{\vec{G}}{r}[\T_{/e/f}\cup \T_{/e/g}]$};
    \end{tikzpicture}
    \caption{Hamiltonian path if $\tilde B_1 \in \FD{\vec{G}}{r}[\T_{/e/f} \cup
      \T_{/e/g} \setminus \{A',A''\}]$.}
    \label{fig:case2}
  \end{figure}
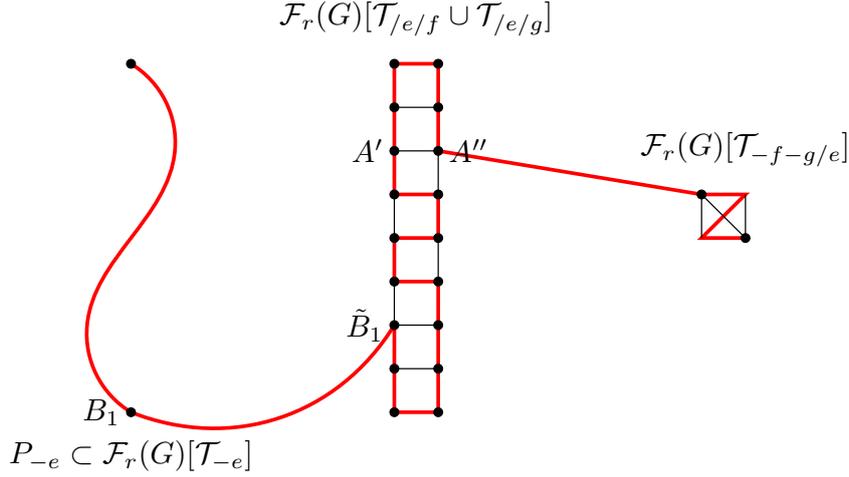

  \subsubsection{$\tilde B_1$ and $\tilde B_2$ belong to $\{A',A''\}$.}\label{subsubsec:2}
  Recall that $\vec G/e$ is of the form described by \ref{it:flip_clique4}
  of \cref{lem:flip_clique} because $A'= A_1'$ and $A'' = A_1''$ for all
  $A \in \T_{-f-g/e}$. Recall also that $w$ is the vertex of $G/e$ resulting
  from the contraction of $r$ and $u$. We say that a vertex of $G$ has \defin{depth} $i$
  if it lies at distance $i$ from $w$ in $\vec G-f-g/e$ (that is  at distance
  $i$ from $\{r,u\}$ in $\vec G-f-g$). We label the vertices in
  $V(\vec G)\setminus\{r,u,v\}$ as in \cref{lem:flip_clique}: let $v_i$ be the
  only vertex of $V(\vec G-f-g/e) \setminus \{w,v\}$ at depth $i$ in the BFS
  starting from the root $w$ in $G-f-g/e$.

  \paragraph{Case 1: $\tilde B_1 \neq \tilde B_2$.}
  Without loss of generality, assume that $\tilde B_1 = A'$ and $\tilde B_2 = A''$.
  \begin{claim}\label{cl:out-neighbourhood1}
    The outneighbourhood of $r$ in $\vec G$ is composed of $u$ which has depth 0, $v_1$
    which has depth one, and $v$ which has depth at least two. 
  \end{claim}
  \begin{poc}
    By definition, $u$ is an outneighbour of $r$ and has depth 0 (and thus
    corresponds to $w$ in \cref{fig:flip_clique}). By \ref{it:flip_clique3} of
    \cref{lem:flip_clique}, the depth of $v$ is at least two. On the other hand,
    other outneighbours of $r$ in $\vec G$ are outneighbours of $w$ in
    $\vec G-f-g/e$ and thus have depth one, so
    $\{u,v\} \subseteq N^+(r) \subseteq \{u,v_1,v\}$. If $v_1$ is not an
    outneighbour of $r$ in $\vec G$, then $r$ has outdegree 1 in $\vec{G}_{-e}$,
    so all arborescences in $\T_{-e}$ contain the edge $f=rv$. Hence
    $\tilde B_1$ and $\tilde B_2$ both belong to $\T_{/e/f}$ so
    $\tilde B_1 = \tilde B_2 = A'$ which is a contradiction.
  \end{poc}

  Assume that $uv_1 \notin E(\vec G)$. Then $v_1u \in E(G)$ because the support of $G-r$ is a clique. Recall that $(u,v)$ was chosen  to maximise $N^+(u)$ among all couples $(u,v) \in N^+(r)^2$ with $uv \in E(\vec G)$, so $|N^+_G(u)| \ge |N^+_G(v_1)| \ge 2$ because $\{u, v_2\} \subseteq N^+_G(v_1)$.  Hence, there exists $x$, an outneighbour of $u$ in
  $\vec G$ different from $v$. We have $x \in N^+_{\vec G-f-g/e}(w)$, hence
  $x$ has depth one, that is $x = v_1$, a contradiction. 
  
  So
  $uv_1 \in E(\vec G)$ and by \cref{cl:out-neighbourhood1}, $N^+_{\vec G}(u) = \{v_1,v\}$. 
 Thus $v_1$ is an outneighbour
  of both $r$ and $u$. Consider the arborescence $A$ of $\vec G-f-g/e$ obtained
  by performing a BFS staring at $w$. As $v_1$ has depth one, $wv_1 \in E(A)$ and there are two arborescences $A_1$ and
  $A_2$ of $\vec G$ that are mapped to $A$ when contracting $e$: one contains
  the arcs $ru$ and $rv_1$, the other contains the arcs $ru$ and $uv_1$. So both
  of them contain $e$ but contain neither $f$ nor $g$. However, flipping $f$
  in them results in two different arborescences: as $v \neq v_1$, the arc
  $uv_1$ is contained in $A_1'$ but not in $A_2'$.  This
  contradicts the fact that $A'$ is constant over $\T_{-f-g/e}$. Hence we cannot have  $\tilde B_1 \neq \tilde B_2$.

  \paragraph{Case 2: $\tilde B_1 = \tilde B_2$.}
  Since $\tilde B_1 = \tilde B_2$, the arborescences $B_1$ and $B_2$ are
  adjacent and $P_{-e}$ is in fact a Hamiltonian cycle. We now proceed similarly
  as after the deduction that $\FD{\vec G}r[\T_{-f-g/e}]$ contains a
  Hamiltonian cycle. Either $\tilde B = \tilde B_1$ for every arborescence $B$ in
  $\T_{-e}$ or there are two arborescences $B_3$ and $B_4$, consecutive on
  $P_{-e}$, such that $\tilde B_3 \neq \tilde B_4$. In the second case,
  $\FD{\vec{G}}{r}[\T_{-e}]$ contains a Hamiltonian path whose extremities are
  $B_3$ and $B_4$, so we are back in the case of one of the previous
  sections (\cref{subsubsec:1} or \cref{subsubsec:2}) or in the case of
  the previous paragraph.

  Thus we can assume that for every $B \in \T_{-e}$, we have
  $\tilde B = \tilde B_1 \in \{A',A''\}$ and thus that
  $\FD{\vec{G}}{r}[\T_{-e}]$ is a clique. \cref{lem:flip_clique} can now be
  applied to $\vec G$ with $w:=r$ and $v:= u$ and to $\vec G/e$ with
  $w:=w$ and $v:=v$ to describe precisely the structure of $\vec G$.
  
  \begin{claim}\label{cl:out-neighbourhood2}
    Either $G$ has a Hamiltonian path, or the outneighbourhood of $r$ in $\vec G$ consists of $u$ which has depth zero, and $v$ which has depth three. Moreover, all arborescences of $\T_{-e}$ contain $f$.
  \end{claim}
  \begin{poc}
    We proceed similarly as in \cref{cl:out-neighbourhood1}. By
    \ref{it:flip_clique4} of \cref{lem:flip_clique} applied to
    $\vec G/e$ with $w:= w$ and $v:=v$, we have
    $\{u,v\} \subseteq N^+_{\vec G}(r) \subseteq N^+_{\vec G-f-g/e}(w) \cup
    \{u\} \subseteq \{u,v, v_1\}$ and $u$ has depth 0 while $v$ has depth at
    least 2 (by \ref{it:flip_clique3}).

    If $v$ has depth at least four, then by \ref{it:flip_clique4} of
    \cref{lem:flip_clique} applied to $\vec{G}/e$ with $w:=w$ and $v:=v$,
    the graph $\vec{G}/e$ has two paths from the root $w$ to $v_1$:
    $w,v, v_1$ and $w, v, v_2, v_3, v_1$. Since $v \in N^+_{\vec G}(r)$, the
    graph $\vec G-e$ also contains two paths from the root $r$ to $v_1$:
    $r,v, v_1$ and $r, v, v_2, v_3, v_1$. By \ref{it:flip_clique2} of
    \cref{lem:flip_clique} applied to $\vec G$ with $w:=r$ and $u:=v$, this
    contradicts the fact that $\tilde B$ is constant for $B \in \T_{-e}$. So the
    depth of $v$ is either two or three.
    
    As $f$ can always be flipped in (because it is incident to the root) in $\vec G$, one
    arborescence of $\T_{-e}$ contains $f$, which implies that for every
    $B \in \T_{-e}$, $\tilde B = A'$. So every arborescence of $\T_{-e}$ contains
    $f$, because flipping in $e$ results in $A'$. So $v_1$ is not an
    outneighbour of $r$, otherwise the path $r, v_1, \dots,v_{k-1}, v$ can be
    completed by \cref{lem:completion} into an arborescence that contains neither $e$ nor $f$, a
    contradiction, which proves that $N^+_{\vec G}(r) = \{u,v\}$.

    We can now prove that $v$ has depth exactly three. Assume towards
    contradiction that $v$ has depth two. By \cref{lem:flip_clique} applied to $\vec G/e$ with $w:=w$ and $v:=v$, the only possibilities for the subgraph of $G$ induced by $\{r,u,v,v_1,v_2\}$ are those drawn on \cref{fig:n2}.
    \begin{figure}[h!]
       \centering
       \begin{subfigure}{.4\textwidth}
       \centering
       \begin{tikzpicture}[decoration={
           markings,
           mark=at position 0.6 with {\arrow[scale=1.5]{>}}}
         ]
         
         \foreach \i in {0,1,2}{
           \coordinate (v\i) at ($(0,{-\i})$); 
         }

         \node[vertex,double] (r) at (1,0){};
         \draw[postaction={decorate}]  (r) -- (v0);
         \draw[postaction={decorate}, bend left] (r) to (v2);
         
         \draw[postaction={decorate}]  (v0) -- (v1);
         \draw[postaction={decorate}] (v1) -- (v2);

         \draw[postaction={decorate}, bend left]  (v0) to (v2);
         \draw[postaction={decorate}, bend left,red]  (v2) to (v1);
         \draw[postaction={decorate}, bend left=40,red]  (v2) to (v0);
         \draw[postaction={decorate}, bend left,red]  (v1) to (v0);

         \foreach \i in {0,...,2}{
           \node[vertex] at (v\i){};
         }
         \node[left = .5cm of v0] {$u$};
         \node[left = .5cm of v1] {$v_1$};
         \node[left = .5cm of v2] {$v=v_2$};
         \node[right = 1pt of r] {$r$};
     \end{tikzpicture}
     \caption{\label{fig:n21}}
     \end{subfigure}
    \hfill
    \begin{subfigure}{.4\textwidth}
    \centering
    \begin{tikzpicture}[decoration={
           markings,
           mark=at position 0.6 with {\arrow[scale=1.5]{>}}}
         ]
    \foreach \i in {0,1,2}{
           \coordinate (v\i) at ($(0,{-\i})$); 
         }
         \node[vertex] (v) at (1,-2){};

         \node[vertex,double] (r) at (1,0){};
         \draw[postaction={decorate}]  (r) -- (v0);
         \draw[postaction={decorate}, bend left] (r) to (v);
         
         \draw[postaction={decorate}]  (v0) -- (v1);
         \draw[postaction={decorate}] (v1) -- (v2);

         \draw[postaction={decorate}]  (v0) to (v);
         \draw[postaction={decorate}, bend left,red]  (v2) to (v1);
         \draw[postaction={decorate}, bend left=40]  (v2) to (v0);
         \draw[postaction={decorate}, bend left,red]  (v1) to (v0);
         \draw[postaction={decorate}]  (v2) to (v);
         \draw[postaction={decorate}]  (v1) to (v);
         \draw[postaction={decorate}, bend right,red]  (v) to (v0);
         \draw[postaction={decorate}, bend left,red]  (v) to (v1);

         \foreach \i in {0,...,2}{
           \node[vertex] at (v\i){};
         }
         \node[left = .5cm of v0] {$u$};
         \node[left = .5cm of v1] {$v_1$};
         \node[left = .5cm of v2] {$v_2$};
         \node[right = .5cm of v] {$v$};
         \node[right = 1pt of r] {$r$};
       \end{tikzpicture}
      \caption{\label{fig:n22}}
       \end{subfigure}
       \caption{The possible graphs $\vec G[r,u,v,v_1,v_2]$ if $v$ has depth. The black arcs
         are contained in $\vec G$, while the red arcs are optional.}
       \label{fig:n2}
     \end{figure}
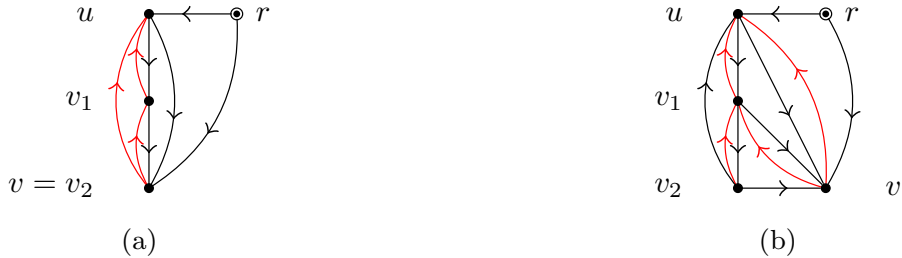
     For any of these possibilities, there are at most two arborescences of $\vec G[r,u,v,v_1,v_2]$ that contain the arcs $e$ and $f$. We also observe from \ref{it:flip_clique4} that for every $3\le i \le n$, $v_{i-1}$
    is the unique inneighbour of $v_i$ that is not one of its descendant. As a result, all arborescences of $G$ contain the path $v_2, \dots v_3$ and $\T_{/e/f}$ contains at most two arborescences. This implies that $\T_{/e/f} \cup T_{/e/g}$ is an edge or a 4-cycle. In either case, $\T$ contains a Hamiltonian path that first visits $\T_{-e}$, then $\T_{/e/f}\cup \T_{/e/g}$ and finally $\T_{-f-g/e}$.
    \end{poc}

  We first recall what we know about the outneighbourhoods of $r$, $u$, $v$ and
  $v_1$. We have $N^+(r) = \{u,v\}$ by \cref{cl:out-neighbourhood2},
  $N^+(u) = \{v_1, v\}$ because $v_1$ has depth one and $v_1 \notin N^+(r)$.  By
  \ref{it:flip_clique4} of \cref{lem:flip_clique} applied to $G/e$ with
  $w:=w$ and $v:=v$, we have $\{v_1\} \subseteq N^+(v) \subseteq \{u, v_1,v_2\}$
  and $N^+(v_1) = \{v_2\}$. So $N^+(\{r,u,v,v_1\}) \cap (\{v_i : i \ge 3\}\setminus \{v\}) = \emptyset$ and
  in every arborescence of $\T_{/e/g}$, there are only two possibilies for the
  path from $r$ to $v_2$: either $r, u, v, v_2$ or $r, u, v, v_1,v_2$. By
  \ref{it:flip_clique4} of \cref{lem:flip_clique} applied to $G/e$ with
  $w:=w$ and $v:=v$ again, for all $i \ge 3$,
  $N^-(v_{i}) \cap (\{r,u,v\} \cup \{v_j : j <i\}) = \{v_{i-1}\}$, a
  straightforward induction shows that in every arborescence of $\T_{/e/g}$, there are
  only two possibilities for the path from $r$ to $v_i$:
  $r, u, v, v_2, \dots, v_i$ or $r, u, v, v_1,v_2, \dots, v_i$. As a result,
  there are at most two arborescences in $\T_{/e/g}$. This implies that
  $\FD{\vec G}r[\T_{/e/f} \cup \T_{/e/g}]$ is either an edge or a 4-cycle. In
  either case, there is a Hamiltonian path $\FD{\vec G}r$, that first visits
  $\T_{-e}$, then $\T_{/e/f}$ (because by \cref{cl:out-neighbourhood2} all
  arborescences of $\T_{-e}$ contain $f$, so flipping in $e$ gives $A'$), then
  $\T_{/e/g}$ and finally $\T_{-f-g/e}$.
\end{proof}

\section*{Acknowledgements}
This work was initiated during the workshop Order \& Geometry 2024 (see
\url{https://sites.google.com/view/ordergeometry2024}).
The third author was supported by the Polish National Science Centre under grant
number UMO-2019/34/E/ST6/00443 and UMO-2023/05/Y/ST6/00079 within the WEAVE-UNISONO program.

\bibliographystyle{alpha}
\bibliography{flip-trees}

\end{document}